\documentclass[oneside,british]{amsart}
\usepackage[T1]{fontenc}
\usepackage[latin9]{inputenc}
\usepackage{amstext}
\usepackage{amsthm}
\usepackage{amssymb}

\makeatletter
\numberwithin{equation}{section}
\numberwithin{figure}{section}
\theoremstyle{plain}
\newtheorem{thm}{\protect\theoremname}[section]
\theoremstyle{definition}
\newtheorem{defn}[thm]{\protect\definitionname}
\theoremstyle{plain}
\newtheorem{cor}[thm]{\protect\corollaryname}
\theoremstyle{plain}
\newtheorem{lem}[thm]{\protect\lemmaname}
\theoremstyle{remark}
\newtheorem*{rem*}{\protect\remarkname}
\theoremstyle{plain}
\newtheorem{prop}[thm]{\protect\propositionname}

\makeatother

\usepackage{babel}
\providecommand{\corollaryname}{Corollary}
\providecommand{\definitionname}{Definition}
\providecommand{\lemmaname}{Lemma}
\providecommand{\propositionname}{Proposition}
\providecommand{\remarkname}{Remark}
\providecommand{\theoremname}{Theorem}

\begin{document}

\title{Dimension of self-conformal measures associated to an exponentially
separated analytic IFS on $\mathbb{R}$}

\author{\noindent Ariel Rapaport}

\subjclass[2000]{\noindent 28A80, 37C45.}

\keywords{self-conformal measure, dimension of measures, exponential separation,
analytic IFS}

\thanks{This research was supported by the Israel Science Foundation (grant
No. 619/22). The author is a Horev Fellow at the Technion -- Israel
Institute of Technology.}
\begin{abstract}
We extend Hochman's work on exponentially separated self-similar measures
on $\mathbb{R}$ to the real analytic setting. More precisely, let
$\Phi=\left\{ \varphi_{i}\right\} _{i\in\Lambda}$ be an iterated
function system on $I:=[0,1]$ consisting of real analytic contractions,
let $p=(p_{i})_{i\in\Lambda}$ be a positive probability vector, and
let $\mu$ be the associated self-conformal measure. Suppose that
the maps in $\Phi$ do not have a common fixed point, $0<\left|\varphi_{i}'(x)\right|<1$
for $i\in\Lambda$ and $x\in I$, and $\Phi$ is exponentially separated.
Under these assumptions, we prove that $\dim\mu=\min\left\{ 1,H(p)/\chi\right\} $,
where $H(p)$ is the entropy of $p$ and $\chi$ is the Lyapunov exponent.

The main novelty of our work lies in an argument that reduces convolutions
of $\mu$ with measures on the (infinite-dimensional) space of real
analytic maps to convolutions with measures on vector spaces of polynomials
of bounded degree. The reason for this reduction is that, for the
latter convolutions, we can establish an entropy increase result,
which plays a crucial role in the proof. We believe that our proof
strategy has the potential to extend other significant recent results
in the dimension theory of stationary fractal measures to the real analytic setting.
\end{abstract}

\maketitle

\section{Introduction and the main result}

\subsection{Background}

Set $I:=[0,1]$, let $\Lambda$ be a finite nonempty index set, and
let $\gamma>0$. For each $i\in\Lambda$, let $\varphi_{i}:I\rightarrow I$
be a $C^{1+\gamma}$-map satisfying $0<\left|\varphi_{i}'(x)\right|<1$
for all $x\in I$. Setting $\Phi:=\left\{ \varphi_{i}\right\} _{i\in\Lambda}$,
the collection $\Phi$ is called a conformal $C^{1+\gamma}$-smooth
iterated function system (IFS) on $I$. By \cite{Hut}, there exists
a unique nonempty compact set $K_{\Phi}\subset I$ such that $K_{\Phi}=\cup_{i\in\Lambda}\varphi_{i}\left(K_{\Phi}\right)$.
It is called the self-conformal set, or the attractor, corresponding
to $\Phi$.

The system $\Phi$ also gives rise to certain natural measures. Let $p=(p_{i})_{i\in\Lambda}$
be a positive probability vector. By \cite{Hut}, there exists a unique
Borel probability measure $\mu$ on $I$ such that $\mu=\sum_{i\in\Lambda}p_{i}\cdot\varphi_{i}\mu$,
where $\varphi_{i}\mu:=\mu\circ\varphi_{i}^{-1}$ is the pushforward of $\mu$ via $\varphi_{i}$.
The measure $\mu$, which is supported on $K_{\Phi}$, is called the
self-conformal measure corresponding to $\Phi$ and $p$.

By the work of Feng and Hu \cite{FH-dimension}, it follows that $\mu$
is exact dimensional. That is, there exists a number $\dim\mu$, called
the dimension of $\mu$, such that
\[
\underset{\delta\downarrow0}{\lim}\frac{\log\mu\left(B(x,\delta)\right)}{\log\delta}=\dim\mu\text{ for }\mu\text{-a.e. }x\in I.
\]
Here, $B(x,\delta)$ denotes the closed ball with centre $x$ and
radius $\delta$. Computing the dimension of self-conformal measures is a natural and
important problem in fractal geometry.

Write $H(p)$ for the entropy of $p$ and $\chi:=\chi\left(\Phi,p\right)$
for the Lyapunov exponent associated to $\Phi$ and $p$. That is,
\[
H(p):=-\sum_{i\in\Lambda}p_{i}\log p_{i}\:\text{ and }\:\chi:=-\sum_{i\in\Lambda}p_{i}\int\log\left|\varphi_{i}'(x)\right|\:d\mu(x),
\]
where we always use $2$ as the base of the $\log$. It is easy to
show that $H(p)/\chi$ is always an upper bound for $\dim\mu$.
Moreover, in the absence of certain obvious obstructions, it is generally expected that
\begin{equation}
\dim\mu=\min\left\{ 1,H(p)/\chi\right\} ,\label{eq:dim mu =00003D min=00007B1,H(p)/chi=00007D}
\end{equation}
although this is often difficult to verify.

The IFS $\Phi$ is said to satisfy the strong separation condition (SSC)
if the sets $\left\{ \varphi_{i}\left(K_{\Phi}\right)\right\} _{i\in\Lambda}$
are disjoint, in which case it is easy to establish (\ref{eq:dim mu =00003D min=00007B1,H(p)/chi=00007D}).
Additionally, under a suitable transversality condition, the equality (\ref{eq:dim mu =00003D min=00007B1,H(p)/chi=00007D})
can be shown to hold almost surely under a natural randomization of
the parameters (see \cite[Section 7]{MR1852098} and \cite[Chapter 14]{MR4661364}). It is desirable
to find explicit conditions for (\ref{eq:dim mu =00003D min=00007B1,H(p)/chi=00007D})
that are much milder than the restrictive SSC.

It is said that $\Phi$ is self-similar if $\varphi_{i}$ is a contracting
similarity for each $i\in\Lambda$. That is, there exist $r_{i},a_{i}\in\mathbb{R}$
such that $0<\left|r_{i}\right|<1$ and $\varphi_{i}(x)=r_{i}x+a_{i}$
for $x\in I$. In recent years, major progress has been made on the
above problem in the self-similar case, initiated by Hochman's seminal
work \cite{Ho1}. Given a word $i_{1}...i_{n}=u\in\Lambda^{n}$, set
$\varphi_{u}:=\varphi_{i_{1}}\circ...\circ\varphi_{i_{n}}$. Also,
let $\Vert\cdot\Vert_{I}$ be the supremum norm on $I$. That is,
$\Vert f\Vert_{I}=\sup_{x\in I}\left|f(x)\right|$ for $f\in C(I)$.
\begin{defn}
\label{def:exp sep}We say that $\Phi$ is exponentially separated
if there exists $c>0$ such that for infinitely many $n\in\mathbb{Z}_{>0}$,
\[
\Vert\varphi_{u_{1}}-\varphi_{u_{2}}\Vert_{I}\ge c^{n}\text{ for all distinct }u_{1},u_{2}\in\Lambda^{n}.
\]
\end{defn}

The main result of \cite{Ho1} states that (\ref{eq:dim mu =00003D min=00007B1,H(p)/chi=00007D})
holds whenever $\Phi$ is self-similar and exponentially separated.
It is a challenging and important open problem to relax these assumptions
as much as possible.

Regarding the exponential separation assumption, note that it is necessary
to assume some form of separation. Indeed, if $\Phi$ does not generate
a free semigroup, in which case it is said to have exact overlaps,
and $\dim\mu<1$, then (\ref{eq:dim mu =00003D min=00007B1,H(p)/chi=00007D})
necessarily fails. One of the most important open problems in fractal
geometry, called the exact overlaps conjecture, asserts that in the
self-similar setting, (\ref{eq:dim mu =00003D min=00007B1,H(p)/chi=00007D})
can only fail in the presence of exact overlaps (see \cite{MR3966837,Var_ICM}
for further discussion). The exact overlaps conjecture has recently
been verified in various situations (see \cite{FengFeng,Rap-EO,rapaport20203maps,Var-Bernoulli}).

Up until the present paper, the self-similarity assumption had not
been strictly relaxed. By the work of Hochman and Solomyak \cite{HS},
equality (\ref{eq:dim mu =00003D min=00007B1,H(p)/chi=00007D}) holds whenever
$\Phi$ is exponentially separated, its maps are restrictions of Möbius
transformations, and the semigroup generated by $\Phi$, considered
as a subsemigroup of $\mathrm{SL}(2,\mathbb{R})$, is strongly irreducible
and proximal\footnote{The main result of \cite{HS} is actually more general, as it does
not require $\mu$ to be compactly supported.}. Aside from this result, \cite{Ho1} had not been extended to other
families of IFSs on the real line. In particular, until the present
work, in all cases where (\ref{eq:dim mu =00003D min=00007B1,H(p)/chi=00007D})
was verified under the exponential separation assumption, the IFS
$\Phi$ was contained in a finite-dimensional Lie group.

\subsection{The main result}

In this paper, we make significant progress on extending \cite{Ho1}
beyond the self-similar setting, and even beyond the situation in
which the acting semigroup is contained in a finite-dimensional manifold.
The following theorem is our main result. Recall that a map $f:I\rightarrow\mathbb{R}$
is said to be real analytic if for each $x_{0}\in I$ there exist
$\epsilon>0$ and $a_{0},a_{1},...\in\mathbb{R}$ such that $f(x)=\sum_{k\ge0}a_{k}(x-x_{0})^{k}$
for all $x\in I\cap B(x_{0},\epsilon)$. When the maps in $\Phi$
are all real analytic, we call it an analytic IFS.
\begin{thm}
\label{thm:main thm}Let $\Phi:=\left\{ \varphi_{i}\right\} _{i\in\Lambda}$
be an analytic IFS on $I$ such that $0<\left|\varphi_{i}'(x)\right|<1$
for all $i\in\Lambda$ and $x\in I$. Suppose that the maps in $\Phi$
do not have a common fixed point, and that $\Phi$ is exponentially
separated. Let $p=(p_{i})_{i\in\Lambda}$ be a positive probability
vector, and let $\mu$ be the self-conformal measure associated to
$\Phi$ and $p$. Then, $\dim\mu=\min\left\{ 1,H(p)/\chi\left(\Phi,p\right)\right\} $.
\end{thm}

Let us make some remarks on the conditions appearing in the theorem.
Firstly, the requirement that $0<\left|\varphi_{i}'(x)\right|<1$
for all $i\in\Lambda$ and $x\in I$ is standard in the context of
conformal IFSs. It ensures that the action is uniformly contracting
and of bounded distortion (see Lemma \ref{lem:bd distort}).

Secondly, the no common fixed point assumption is equivalent to $K_{\Phi}$
not being a singleton. This is clearly necessary for the validity
of (\ref{eq:dim mu =00003D min=00007B1,H(p)/chi=00007D}), unless
$\Phi$ consists of a single map. This assumption does not appear
in \cite{Ho1}, as exponential separation makes it redundant in the
self-similar setting.

Thirdly, as in the self-similar case, even though the exponential
separation condition is quite mild, it is desirable to relax it to
the necessary assumption of no exact overlaps. Given that the exact
overlaps conjecture is still completely open even in the self-similar
setting, at present this is well beyond our reach.

Finally, the analyticity assumption is crucial for our proof of Theorem
\ref{thm:main thm}. Indeed, in the proof, we assume by contradiction
that
\begin{equation}
\Delta:=\min\left\{ 1,H(p)/\chi\right\} -\dim\mu>0,\label{eq:def of Delta in intro}
\end{equation}
and consider $k$-th order Taylor polynomials at different points
in $K_{\Phi}$ of elements in the semigroup $\mathrm{S}_{\Phi}:=\left\{ \varphi_{u}\right\} _{u\in\Lambda^{*}}$
generated by $\Phi$. The smaller the values of $\Delta$ and $c$,
where $c>0$ is the exponential separation constant from Definition
\ref{def:exp sep}, the larger we need to take $k$. Furthermore,
using analyticity and bounded distortion, we obtain upper bounds on
the higher derivatives of members of $\mathrm{S}_{\Phi}$ in terms
of their first derivative (see Lemma \ref{lem:bd distort}), which
are essential for the proof. Consequently, we cannot relax the analyticity
condition to merely assume that the maps in $\Phi$ are of class $C^{\infty}$.

One of the main advantages of \cite{Ho1} is that it provides many
new explicit examples in which (\ref{eq:dim mu =00003D min=00007B1,H(p)/chi=00007D})
holds. Indeed, in the self-similar setting, when $\Phi$ is defined
by algebraic parameters, it has no exact overlaps if and only if it
is exponentially separated. Unfortunately, we are unable to prove
such a statement in the analytic case at present. Even when $\Phi$
consists of polynomials with algebraic coefficients, it is unclear to us
whether $\Phi$ is exponentially separated whenever it generates a
free semigroup. The problem lies in the fact that, when $\Phi$ consists
of polynomials, the maximal degree of a polynomial in $\{\varphi_{u}\}_{u\in\Lambda^{n}}$
increases exponentially with $n$, unless $\Phi$ is self-similar.

At present, the main benefit of Theorem \ref{thm:main thm} is that
it imposes strong restrictions on the system $\Phi$ whenever (\ref{eq:dim mu =00003D min=00007B1,H(p)/chi=00007D})
fails. Furthermore, as described in Section \ref{subsec:one-param families} below, we can use Theorem \ref{thm:main thm}, in combination with a result of Solomyak and Takahashi \cite{MR4300232}, to show that, under mild assumptions, given a one-parameter family of analytic IFSs, the set of parameters for which (\ref{eq:dim mu =00003D min=00007B1,H(p)/chi=00007D})
fails has zero Hausdorff dimension. Finally, we believe that, in the future, our proof strategy could be adapted to extend other significant recent results in the dimension theory of stationary fractal measures to the real analytic setting.

\subsection{Dimension of self-conformal sets}

Using Theorem \ref{thm:main thm}, we are also able to compute the
Hausdorff dimension of $K_{\Phi}$ in the analytic exponentially separated
case. 

For each $t\ge0$ set
\[
P_{\Phi}(t):=\underset{n\rightarrow\infty}{\lim}\frac{1}{n}\log\left(\sum_{u\in\Lambda^{n}}\Vert\varphi_{u}'\Vert_{I}^{t}\right),
\]
where the limit exists by sub-additivity. It is easy to verify that
$P_{\Phi}(0)=\log|\Lambda|$, $P_{\Phi}(t)\rightarrow-\infty$ as
$t\rightarrow\infty$, and $P_{\Phi}$ is strictly decreasing and
continuous. Thus, there exists a unique nonnegative real number $s(\Phi)$ such that $P_{\Phi}\left(s(\Phi)\right)=0$.
Following \cite{MR4661364}, we call $s(\Phi)$ the conformal
similarity dimension associated to $\Phi$.

It is not difficult to see that $\min\left\{ 1,s(\Phi)\right\} $
is always an upper bound for $\dim_{H}K_{\Phi}$, and that $\dim_{H}K_{\Phi}=s(\Phi)$
whenever $\Phi$ satisfies the SSC. Moreover, under a suitable transversality condition, the equality $\dim_{H}K_{\Phi}=\min\left\{ 1,s(\Phi)\right\} $ holds
almost surely under a natural randomization of the parameters. For the proof of these facts and for further
discussion, we refer the reader to \cite[Chapter 14]{MR4661364}.
\begin{cor}
\label{cor:main cor for sets}
Let $\Phi:=\left\{ \varphi_{i}\right\} _{i\in\Lambda}$ be an analytic
IFS on $I$ such that $0<\left|\varphi_{i}'(x)\right|<1$ for all
$i\in\Lambda$ and $x\in I$. Suppose that the maps in $\Phi$ do
not have a common fixed point, and that $\Phi$ is exponentially separated.
Then $\dim_{H}K_{\Phi}=\min\left\{ 1,s(\Phi)\right\} $.
\end{cor}

\begin{proof}
Let $\sigma:\Lambda^{\mathbb{N}}\rightarrow\Lambda^{\mathbb{N}}$
be the left-shift map. Using \cite[Theorem 1.22]{MR2423393}, it can
be shown that there exists a unique $\sigma$-invariant Borel probability
measure $\nu$ on $\Lambda^{\mathbb{N}}$ such that $s(\Phi)=h(\nu)/\chi(\Phi,\nu)$.
Here $h(\nu)$ is the entropy of $\nu$, and $\chi(\Phi,\nu)$ is
the Lyapunov exponent associated to $\Phi$ and $\nu$ (see \cite[Chapter 14]{MR4661364}).
Moreover, by the construction of $\nu$ (again see \cite{MR2423393}),
we have $\nu\left([u]\right)>0$ for each $u\in\Lambda^{*}$, where $[u]$
is the cylinder set determined by $u$ (see Section \ref{subsec:Symbolic-notations}
below).

Given $n\ge1$, set $\Phi_{n}:=\left\{ \varphi_{u}\right\} _{u\in\Lambda^{n}}$
and $q_{n}:=\left(\nu\left([u]\right)\right)_{u\in\Lambda^{n}}$.
It is easy to verify that $\frac{1}{n}H(q_{n})\overset{n}{\rightarrow}h(\nu)$
and $\frac{1}{n}\chi(\Phi_{n},q_{n})\overset{n}{\rightarrow}\chi(\Phi,\nu)$,
and so
\begin{equation}
\underset{n\rightarrow\infty}{\lim}\left(H(q_{n})/\chi(\Phi_{n},q_{n})\right)=s(\Phi).\label{eq:conv of ratio to s(Phi)}
\end{equation}
For $n\ge1$, let $\mu_{n}$ be the self-conformal measure corresponding
to $\Phi_{n}$ and $q_{n}$. Since $\Phi$ is exponentially separated,
it is easy to see that this is also the case for $\Phi_{n}$. Hence,
by Theorem \ref{thm:main thm} and since $\mu_{n}$ is supported on
$K_{\Phi}$,
\[
\min\left\{ 1,H(q_{n})/\chi(\Phi_{n},q_{n})\right\} =\dim\mu_{n}\le\dim_{H}K_{\Phi}.
\]
This together with (\ref{eq:conv of ratio to s(Phi)}) completes the
proof of the corollary.
\end{proof}

\subsection{\label{subsec:one-param families}One-parameter families of analytic IFSs}

As mentioned above, our results can be combined with a result of Solomyak and Takahashi \cite{MR4300232} to deduce a statement about one-parameter families of analytic IFSs.

For each $i\in\Lambda$, let $\{\varphi_{i,t}\}_{t\in I}$ be a one-parameter family of maps from $I$ into the interval $(0,1)$. Suppose that the map $(t, x) \mapsto \varphi_{i,t}(x)$ is real analytic on $I^2$ and satisfies $0<\left|\frac{d}{dx}\varphi_{i,t}(x)\right|<1$ for all $(t,x)\in I^{2}$.
For $t\in I$, set $\Phi_{t}:=\left\{ \varphi_{i,t}\right\} _{i\in\Lambda}$ and let $K_{\Phi_{t}}$ be the self-conformal set corresponding to $\Phi_{t}$. Given a probability vector $p=(p_{i})_{i\in\Lambda}$,
write $\mu_{t,p}$ for the self-conformal measure associated to $\Phi_{t}$
and $p$.

Given $\omega, \eta \in \Lambda^{\mathbb{N}}$, let $F_{\omega,\eta}:I \to \mathbb{R}$ be defined by  
\[
F_{\omega,\eta}(t):=\underset{n\rightarrow\infty}{\lim}\left(\varphi_{\omega_{0},t}\circ...\circ\varphi_{\omega_{n},t}(0)-\varphi_{\eta_{0},t}\circ...\circ\varphi_{\eta_{n},t}(0)\right)\text{ for }t\in I.
\]  
We say that the family of IFSs $\left\{ \Phi_{t}\right\} _{t\in I}$ is non-degenerate if, for all distinct $\omega, \eta \in \Lambda^{\mathbb{N}}$, the function $F_{\omega,\eta}$ is not identically zero.

\begin{cor}
\label{cor:param families}Suppose that $|\Lambda|\ge2$ and that
$\left\{ \Phi_{t}\right\} _{t\in I}$ is non-degenerate. Then there
exists $E\subset I$ with $\dim_{H}E=0$ such that for all $t\in I\setminus E$
we have $\dim_{H}K_{\Phi_{t}}=\min\left\{ 1,s(\Phi_{t})\right\} $
and $\dim\mu_{t,p}=\min\left\{ 1,H(p)/\chi\left(\Phi_{t},p\right)\right\} $
for every positive probability vector $p=(p_{i})_{i\in\Lambda}$.
\end{cor}

\begin{rem*}
In the self-similar case, Hochman \cite{Ho1} obtained a stronger statement in which the exceptional set $E$ is of zero packing dimension.
\end{rem*}

\begin{proof}
Let $E_{1}$ be the set of $t\in I$ for which $\Phi_{t}$ is not
exponentially separated, and let $E_{2}$ be the set of $t\in I$
for which the maps in $\Phi_{t}$ have a common fixed point.
By Theorem \ref{thm:main thm} and Corollary \ref{cor:main cor for sets},
in order to prove the corollary it suffices to show that $\dim_{H}E_{1}=\dim_{H}E_{2}=0$.
Since $\left\{ \Phi_{t}\right\} _{t\in I}$ is non-degenerate, it follows by \cite[Theorem 2.10]{MR4300232} that $\dim_{H}E_{1}=0$.
We next complete the proof by showing that $E_{2}$ is finite.

Let $i_{1},i_{2}\in\Lambda$ be distinct, for $j=1,2$ let $i_{j}^{\infty}$
denote the infinite word in $\Lambda^{\mathbb{N}}$ consisting only
of the letter $i_{j}$, and set $F:=F_{i_{1}^{\infty},i_{2}^{\infty}}$.
Given $t\in I$, note that $\varphi_{i_{1},t}$ and $\varphi_{i_{2},t}$
share a common fixed point if and only if $F(t)=0$. Consequently,
we have $E_{2}\subset F^{-1}\{0\}$. Since $\left\{ \Phi_{t}\right\} _{t\in I}$
is non-degenerate, $F$ is not identically
zero. Moreover, by \cite[Lemma 2.6]{MR4300232}, $F$ is real analytic
on $I$. Together, these two facts imply that $F^{-1}\{0\}$ is finite,
completing the proof of the corollary.
\end{proof}
It is easy to construct explicit families of IFSs for which Corollary
\ref{cor:param families} applies. In particular, as noted in
\cite[Corollary 2.11]{MR4300232}, if there exists $t_{0}\in I$ such
that $\Phi_{t_{0}}$ satisfies the SSC, then $F_{\omega,\eta}(t_{0})\ne0$
for all distinct $\omega,\eta\in\Lambda^{\mathbb{N}}$. Thus, in this case, $\left\{ \Phi_{t}\right\} _{t\in I}$ is non-degenerate, and Corollary \ref{cor:param families} applies.

\subsection{\label{subsec:About-the-proof}About the proof and structure of the
paper}

The main difficulty in proving Theorem \ref{thm:main thm} lies in
the fact that $\mathcal{A}(I)$, the vector space of real analytic
maps from $I$ to $\mathbb{R}$, is infinite-dimensional. This creates
significant obstacles when attempting to adapt the argument in \cite{Ho1}
to the present setting.

To address these obstacles, in a certain sense we reduce the problem
to a finite-dimensional situation. This is achieved by using Taylor's
theorem to approximate convolutions of the form $\nu.\mu$, where
$\mu$ is self-conformal and $\nu$ is a probability measure on $\mathcal{A}(I)$,
with convolutions of the form $\nu'.\mu$, where $\nu'$ is a probability
measure on the vector space $\mathcal{P}_{k}$ of polynomials in $\mathbb{R}[X]$
of degree at most $k$. Here, $k$ is an integer assumed to be large
with respect to $c$ from Definition \ref{def:exp sep} and $\Delta$
given in (\ref{eq:def of Delta in intro}).

The advantage of working with $\nu'.\mu$ in place of $\nu.\mu$ is
that, for the former convolutions, we can use Hochman's inverse theorem
from \cite{Ho1} to establish an entropy increase result. Roughly
speaking, it states that if $\nu'$ has nonnegligible entropy and
$\dim\mu<1$, then $\nu'.\mu$ has significantly more entropy than
$\mu$. Using this result, along with the above approximation of $\nu.\mu$
by $\nu'.\mu$, we derive the desired contradiction whenever $\Delta>0$,
which proves the theorem.

The proof of the entropy increase result follows the lines of analogous
statements for convolutions of self-similar and self-affine measures in higher dimensions with measures on spaces of affine maps (see \cite{BHR,Ho}). The main novelty of the present work lies in
the above reduction to a situation where the acting space of maps is finite-dimensional.

The rest of this paper is organized as follows. In Section \ref{sec:Preliminaries},
we introduce necessary notations and background. In Section \ref{sec:An-entropy-increase-result},
we establish the aforementioned entropy increase result. In Section
\ref{sec:Proof-of-the-main-result}, we carry out the proof of Theorem
\ref{thm:main thm}.

\subsubsection*{\textbf{\emph{Acknowledgment}}}

I would like to thank Boris Solomyak for suggesting the application for parametric families of IFSs discussed in Section \ref{subsec:one-param families}. I would also like to thank Zhou Feng and Haojie Ren for helpful remarks on an earlier version of this paper.

\section{\label{sec:Preliminaries}Preliminaries}

\subsection{\label{subsec:Basic-notations}Basic notations}

In what follows, the base of the logarithm is always $2$.

For a metric space $X$, denote by $\mathcal{M}(X)$ the collection
of all compactly supported Borel probability measures on $X$. Given
another metric space $Y$, a Borel map $f:X\rightarrow Y$, and a
measure $\nu\in\mathcal{M}(X)$, we write $f\nu:=\nu\circ f^{-1}$
for the pushforward of $\nu$ via $f$. For a Borel set $E\subset X$
with $\nu(E)>0$, we denote by $\nu_{E}$ the conditioning of $\nu$
on $E$. That is, $\nu_{E}:=\frac{1}{\nu(E)}\nu|_{E}$, where $\nu|_{E}$
is the restriction of $\nu$ to $E$.

Given a partition $\mathcal{D}$ of a set $X$, for $x\in X$ we denote
by $\mathcal{D}(x)$ the unique $D\in\mathcal{D}$ containing $x$.

For a real vector space $V$, a scalar $c\in\mathbb{R}$, and a vector
$w\in V$, define $S_{c}:V\rightarrow V$ and $T_{w}:V\rightarrow V$
by $S_{c}v=cv$ and $T_{w}v=w+v$ for $v\in V$.

Given an integer $n\ge1$, set $\mathcal{N}_{n}:=\left\{ 1,...,n\right\} $
and denote the normalized counting measure on $\mathcal{N}_{n}$ by
$\lambda_{n}$. That is, $\lambda_{n}\{i\}=1/n$ for each $ i \in \mathcal{N}_{n}$.

\subsection{Relations between parameters}

Given $R_{1},R_{2}\in\mathbb{R}$ with $R_{1},R_{2}\ge1$, we write
$R_{1}\ll R_{2}$ in order to indicate that $R_{2}$ is large with
respect to $R_{1}$. Formally, this means that $R_{2}\ge f(R_{1})$,
where $f$ is an unspecified function from $[1,\infty)$ into itself.
The values attained by $f$ are assumed to be sufficiently large in
a manner depending on the specific context.

Similarly, given $0<\epsilon_{1},\epsilon_{2}<1$ we write $R_{1}\ll\epsilon_{1}^{-1}$,
$\epsilon_{2}^{-1}\ll R_{2}$ and $\epsilon_{1}^{-1}\ll\epsilon_{2}^{-1}$
in order to respectively indicate that $\epsilon_{1}$ is small with
respect to $R_{1}$, $R_{2}$ is large with respect to $\epsilon_{2}$,
and $\epsilon_{2}$ is small with respect to $\epsilon_{1}$.

The relation $\ll$ is clearly transitive. That is, if $R_{1}\ll R_{2}$
and for $R_{3}\ge1$ we have $R_{2}\ll R_{3}$, then also $R_{1}\ll R_{3}$.
For instance, the sentence `Let $m\ge1$, $k\ge K(m)\ge1$ and $n\ge N(m,k)\ge1$
be given' is equivalent to `Let $m,k,n\ge1$ be with $m\ll k\ll n$'.

\subsection{\label{subsec:The-setup}The setup}

Set $I:=[0,1]$, let $\Lambda$ be a finite nonempty index set, and
for each $i\in\Lambda$ let $\varphi_{i}:I\rightarrow I$ be real
analytic. Suppose moreover that
\begin{equation}
0<\left|\varphi_{i}'(x)\right|<1\text{ for }i\in\Lambda\text{ and }x\in I,\label{eq:der bd from 0,1 asump}
\end{equation}
and set $\Phi:=\left\{ \varphi_{i}\right\} _{i\in\Lambda}$. Let $p=(p_{i})_{i\in\Lambda}$
be a positive probability vector, and let $\mu$ be the self-conformal
measure associated to $\Phi$ and $p$. That is, $\mu$ is the unique
element in $\mathcal{M}(I)$ satisfying the relation
\begin{equation}
\mu=\sum_{i\in\Lambda}p_{i}\cdot\varphi_{i}\mu.\label{eq:def rel of mu}
\end{equation}

As before, denote by $\chi$ the Lyapunov exponent associated to $\Phi$ and
$p$, which is defined by
\[
\chi:=-\sum_{i\in\Lambda}p_{i}\int\log\left|\varphi_{i}'(x)\right|\:d\mu(x).
\]

In what follows, we always assume that the maps in $\Phi$ do not
have a common fixed point. It is easy to see that this implies that
$\mu$ is nonatomic. That is, $\mu\{x\}=0$ for all $x\in I$.

\subsection{Bounded distortion}

The following lemma plays an important role in our arguments. Let
$\Lambda^{*}$ be the set of finite words over $\Lambda$. Recall
that for $i_{1}...i_{n}=u\in\Lambda^{*}$ we write $\varphi_{u}$
in place of the composition $\varphi_{i_{1}}\circ...\circ\varphi_{i_{n}}$.
\begin{lem}
\label{lem:bd distort}There exists a constant $C>1$ such that,
\begin{enumerate}
\item \label{enu:bd prop}$\left|\varphi_{u}'(x)\right|\le C\left|\varphi_{u}'(y)\right|$
for all $u\in\Lambda^{*}$ and $x,y\in I$;
\item $\left|\varphi_{u}^{(k)}(x)\right|\le k!C^{k}\left|\varphi_{u}'(0)\right|$
for all $u\in\Lambda^{*}$, $k\in\mathbb{Z}_{>0}$ and $x\in I$,
where $\varphi_{u}^{(k)}$ denotes the $k$th derivative of $\varphi_{u}$.
\end{enumerate}
\end{lem}

\begin{rem*}
Part (\ref{enu:bd prop}) of the lemma is often referred to as the
bounded distortion property.
\end{rem*}
\begin{proof}
The proof of the bounded distortion property is standard, and can
be found for instance in \cite[Chapter 14]{MR4661364}. Let us prove
the second part of the lemma.

For $\epsilon>0$, write $I^{(\epsilon)}$ for the open $\epsilon$-neighbourhood
of $I$ in $\mathbb{C}$. That is,
\[
I^{(\epsilon)}:=\left\{ z\in\mathbb{C}\::\:|z-x|<\epsilon\text{ for some }x\in I\right\} .
\]
By our assumptions on $\Phi$, there exists $0<\epsilon<1$ such that,
for each $i\in\Lambda$, the map $\varphi_{i}:I\rightarrow I$ extends
to a holomorphic map from $I^{(\epsilon)}$ to $\mathbb{C}$, satisfying
\begin{equation}
\epsilon\le\left|\varphi_{i}'(z)\right|\le1-\epsilon\:\text{ and }\:\left|\varphi_{i}''(z)\right|\le\epsilon^{-1}\:\text{ for }z\in I^{(\epsilon)}.\label{eq:prop nec for bd dist}
\end{equation}
Given $z\in I^{(\epsilon)}$, there exists $x\in I$ with $|z-x|<\epsilon$.
Since $\varphi_{i}(x)\in I$ and
\[
\left|\varphi_{i}(z)-\varphi_{i}(x)\right|\le(1-\epsilon)|z-x|<\epsilon,
\]
it follows that $\varphi_{i}\left(I^{(\epsilon)}\right)\subset I^{(\epsilon)}$.
Thus, $\Phi$ may be regarded as a conformal IFS on $I^{(\epsilon)}$. From this,
from (\ref{eq:prop nec for bd dist}), and by the bounded distortion
property (again, see \cite[Chapter 14]{MR4661364}), there exists
$C_{0}>1$ such that
\begin{equation}
\left|\varphi_{u}'(z)\right|\le C_{0}\left|\varphi_{u}'(w)\right|\text{ for all }u\in\Lambda^{*}\text{ and }z,w\in I^{(\epsilon)}.\label{eq:by BD}
\end{equation}

Now, let $u\in\Lambda^{*}$ be given and set $f:=\varphi_{u}-\varphi_{u}(0)$.
By (\ref{eq:by BD}), for $z\in I^{(\epsilon)}$
\[
|f(z)|\le|z|\sup_{w\in I^{(\epsilon)}}\left|\varphi_{u}'(w)\right|\le2C_{0}\left|\varphi_{u}'(0)\right|.
\]
Thus, by Cauchy's estimate, for $k\in\mathbb{Z}_{>0}$ and $x\in I$
\[
\left|\varphi_{u}^{(k)}(x)\right|=\left|f^{(k)}(x)\right|\le k!\epsilon^{-k}\cdot2C_{0}\left|\varphi_{u}'(0)\right|,
\]
which completes the proof of the lemma.
\end{proof}

\subsection{\label{subsec:Function-spaces}Function spaces}

Given $k\ge1$, denote by $\mathcal{P}_{k}$ the vector space of polynomials
$p\in\mathbb{R}[X]$ with $\deg p\le k$. Let $\Vert\cdot\Vert_{2}$
be the norm on $\mathcal{P}_{k}$ such that $\Vert p\Vert_{2}^{2}:=\sum_{j=0}^{k}|a_{j}|^{2}$
for $p(X)=\sum_{j=0}^{k}a_{j}X^{j}\in\mathcal{P}_{k}$. In what follows,
all metric and topological concepts in $\mathcal{P}_{k}$ are considered
with respect to $\Vert\cdot\Vert_{2}$.

Denote by $\mathcal{A}(I)$ the vector space of real analytic maps
from $I$ to $\mathbb{R}$. For $k\ge1$ and $a\in I$, let $P_{k,a}:\mathcal{A}(I)\rightarrow\mathcal{P}_{k}$
be such that $P_{k,a}f$ equals the $k$-th order Taylor polynomial
of $f\in\mathcal{A}(I)$ at the point $a$. That is,
\begin{equation}
P_{k,a}f(X)=\sum_{j=0}^{k}\frac{f^{(j)}(a)}{j!}(X-a)^{j}.\label{eq:exp form of P_k,a}
\end{equation}

Let $\nu$ be a finitely supported probability measure on $\mathcal{A}(I)$
or a member of $\mathcal{M}\left(\mathcal{P}_{k}\right)$ for some
$k\ge1$. Given $\theta\in\mathcal{M}(I)$, we write $\nu.\theta$
for the push-forward of $\nu\times\theta$ via the map sending $(f,x)\in\mathcal{A}(I)\times I$
to $f(x)$. For $x\in I$ we write $\nu.x$ in place of $\nu.\delta_{x}$,
where $\delta_{x}$ is the Dirac mass at $x$.

\subsection{\label{subsec:Symbolic-notations}Symbolic related notations}

Given $n\ge1$ and $(\omega_{k})_{k\ge0}=\omega\in\Lambda^{\mathbb{N}}$,
write $\omega|_{n}:=\omega_{0}...\omega_{n-1}\in\Lambda^{n}$ for
the $n$th prefix of $\omega$. For a word $i_{1}...i_{n}=u\in\Lambda^{n}$,
denote by $[u]\subset\Lambda^{\mathbb{N}}$ the cylinder set determined
by $u$. That is, $[u]:=\left\{ \omega\in\Lambda^{\mathbb{N}}\::\:\omega|_{n}=u\right\} $.
Let $\mathcal{C}_{n}:=\left\{ [u]\right\} _{u\in\Lambda^{n}}$ denote
the partition of $\Lambda^{\mathbb{N}}$ into level-$n$ cylinder
sets. Given a set of words $\mathcal{U}\subset\Lambda^{*}$, we often
write $\left[\mathcal{U}\right]$ in place of $\cup_{u\in\mathcal{U}}[u]$.

Denote by $\beta:=p^{\mathbb{N}}$ the Bernoulli measure on $\Lambda^{\mathbb{N}}$
associated to $p$. That is, $\beta$ is the unique element in $\mathcal{M}\left(\Lambda^{\mathbb{N}}\right)$
such that $\beta\left([u]\right)=p_{u}$ for $u\in\Lambda^{*}$, where
$p_{u}:=p_{i_{1}}\cdot...\cdot p_{i_{n}}$ for $i_{1}...i_{n}=u\in\Lambda^{n}$.
Let $\sigma:\Lambda^{\mathbb{N}}\rightarrow\Lambda^{\mathbb{N}}$
be the left-shift map. From the ergodicity of $\left(\Lambda^{\mathbb{N}},\sigma,\beta\right)$
and by bounded distortion (see Lemma \ref{lem:bd distort}),
\begin{equation}
\chi=-\underset{n\rightarrow\infty}{\lim}\frac{1}{n}\log\left|\varphi_{\omega|_{n}}'(0)\right|\text{ for }\beta\text{-a.e. }\omega.\label{eq:chi as a.s.  limit}
\end{equation}

For $n\ge1$, let $\Psi_{n}$ denote the set of words $i_{1}...i_{k}=u\in\Lambda^{*}$
such that $\left|\varphi_{u}'(0)\right|\le2^{-n}$ and $\left|\varphi_{i_{1}...i_{s}}'(0)\right|>2^{-n}$
for each $1\le s<k$. Note that $\Psi_{n}$ is a minimal cut-set for
$\Lambda^{*}$, which means that for each $\omega\in\Lambda^{\mathbb{N}}$
there exists a unique $u\in\Psi_{n}$ with $\omega\in[u]$. This,
together with (\ref{eq:def rel of mu}), implies that
\begin{equation}
\mu=\sum_{u\in\Psi_{n}}p_{u}\cdot\varphi_{u}\mu\text{ for each }n\ge1.\label{eq:mu=00003DSum_=00007Bu in Psi_n=00007D}
\end{equation}
Moreover, from (\ref{eq:der bd from 0,1 asump}) and by bounded distortion,
there exists a global constant $C>1$ such that
\begin{equation}
C^{-1}2^{-n}\le\left|\varphi_{u}'(x)\right|\le C2^{-n}\text{ for }n\ge1,\:u\in\Psi_{n}\text{ and }x\in I.\label{eq:der for u in Psi_n}
\end{equation}

Let $\Pi:\Lambda^{\mathbb{N}}\rightarrow I$ be the coding map associated
to $\Phi$. That is,
\[
\Pi(\omega):=\underset{n\rightarrow\infty}{\lim}\varphi_{\omega|_{n}}(0)\text{ for }\omega\in\Lambda^{\mathbb{N}}.
\]
Additionally, for $n\ge1$, let $\Pi_{n}:\Lambda^{\mathbb{N}}\rightarrow\mathcal{A}(I)$
be with $\Pi_{n}(\omega)=\varphi_{\omega|_{n}}$ for $\omega\in\Lambda^{\mathbb{N}}$.

Write $\{\beta_{\omega}\}_{\omega\in\Lambda^{\mathbb{N}}}\subset\mathcal{M}\left(\Lambda^{\mathbb{N}}\right)$
for the disintegration of $\beta$ with respect to $\Pi^{-1}\mathcal{B}_{\mathbb{R}}$,
where $\mathcal{B}_{\mathbb{R}}$ is the Borel $\sigma$-algebra of
$\mathbb{R}$. For the definition and basic properties of disintegrations,
we refer the reader to \cite[Theorem 5.14]{EiWa} and the discussion
following it.

\subsection{Entropy}

Let $(X,\mathcal{F})$ be a measurable space. Given a probability
measure $\theta$ on $X$ and a countable partition $\mathcal{D}\subset\mathcal{F}$
of $X$, the entropy of $\theta$ with respect to $\mathcal{D}$ is
defined by
\[
H(\theta,\mathcal{D}):=-\sum_{D\in\mathcal{D}}\theta(D)\log\theta(D).
\]
If $\mathcal{E}\subset\mathcal{F}$ is another countable partition
of $X$, the conditional entropy given $\mathcal{E}$ is defined as
follows
\[
H(\theta,\mathcal{D}\mid\mathcal{E}):=\sum_{E\in\mathcal{E}}\theta(E)\cdot H(\theta_{E},\mathcal{D}).
\]

Throughout the paper, we repeatedly use basic properties of entropy
and conditional entropy. We frequently do so without explicit reference.
Readers are advised to consult \cite[Section 3.1]{Ho1} for details.

\subsection{\label{subsec:Dyadic-partitions}Dyadic partitions and component
measures}

Given $d\ge1$ and $n\ge0$, write $\mathcal{D}_{n}^{d}$ for the
level-$n$ dyadic partition of $\mathbb{R}^{d}$. That is,
\[
\mathcal{D}_{n}^{d}:=\left\{ \Big[\frac{m_{1}}{2^{n}},\frac{m_{1}+1}{2^{n}}\Big)\times...\times\Big[\frac{m_{d}}{2^{n}},\frac{m_{d}+1}{2^{n}}\Big)\::\:m_{1},...,m_{d}\in\mathbb{Z}\right\} .
\]
For a real number $t\ge0$, we write $\mathcal{D}_{t}^{d}$ in place
of $\mathcal{D}_{\left\lfloor t\right\rfloor }^{d}$, where $\left\lfloor t\right\rfloor $
is the integral part of $t$. We often omit the superscript $d$ when
it is clear from the context.

Dyadic entropy, i.e. entropy measured with respect to the partitions
$\mathcal{D}_{n}^{1}$, plays an important role in our arguments.
Note that by \cite{Yo}, when $\theta\in\mathcal{M}\left(\mathbb{R}\right)$
is exact dimensional,
\begin{equation}
\underset{n\rightarrow\infty}{\lim}\frac{1}{n}H\left(\theta,\mathcal{D}_{n}\right)=\dim\theta.\label{eq:ent dim =00003D exact dim}
\end{equation}

We shall also need the following two simple lemmas. They follow easily from
basic properties of entropy, and their proof is therefore omitted.
\begin{lem}
\label{lem:dyad ent =000026 lip func}Let $d\ge1$ and $\theta\in\mathcal{M}\left(\mathbb{R}^{d}\right)$
be given. Let $C>1$, and let $f:\mathrm{supp}(\theta)\rightarrow\mathbb{R}^{d}$
be bi-Lipschitz with bi-Lipschitz constant $C$. That is, $C^{-1}|x-y|\le\left|f(x)-f(y)\right|\le C|x-y|$
for $x,y\in\mathrm{supp}(\theta)$. Then for each $n\ge0$,
\[
H\left(f\theta,\mathcal{D}_{n}\right)=H\left(\theta,\mathcal{D}_{n}\right)+O_{d}\left(1+\log C\right).
\]
\end{lem}

\begin{lem}
\label{lem:cont of ent wrt close maps}Let $(X,\mathcal{F},\theta)$ be a probability space, and let $f,g:X\rightarrow \mathbb{R}$ be measurable. Let $n\ge0$, and suppose that $\left|f(x)-g(x)\right|\le2^{-n}$ for all $x\in X$. Then,
\[
H\left(f\theta,\mathcal{D}_{n}\right)=H\left(g\theta,\mathcal{D}_{n}\right)+O\left(1\right).
\]
\end{lem}

We also need to consider dyadic partitions of vector spaces of polynomials.
For $k\ge1$ and $n\ge0$, we denote by $\mathcal{D}_{n}^{\mathcal{P}_{k}}$
the level-$n$ dyadic partition of $\mathcal{P}_{k}$, which is given
by 
\[
\mathcal{D}_{n}^{\mathcal{P}_{k}}:=\left\{ \left\{ \sum_{j=0}^{k}a_{j}X^{j}\::\:(a_{0},...,a_{k})\in D\right\} \::\:D\in\mathcal{D}_{n}^{k+1}\right\} .
\]
Thus, $\mathcal{D}_{n}^{\mathcal{P}_{k}}$ is defined by identifying
$\mathcal{P}_{k}$ with $\mathbb{R}^{k+1}$ in the natural way. We omit
the superscript $\mathcal{P}_{k}$ when it is clear from the context.

Given $\theta\in\mathcal{M}\left(\mathbb{R}\right)$, $n\ge0$, and
$x\in\mathbb{R}$ with $\theta\left(\mathcal{D}_{n}(x)\right)>0$,
we write $\theta_{x,n}$ is place of the conditional measure $\theta_{\mathcal{D}_{n}(x)}$.
Similarly, for $\nu\in\mathcal{M}\left(\mathcal{P}_{k}\right)$ and
$p\in\mathcal{P}_{k}$ with $\nu\left(\mathcal{D}_{n}(p)\right)>0$,
we write $\nu_{p,n}$ in place of $\nu_{\mathcal{D}_{n}(p)}$. The
measures $\theta_{x,n}$ and $\nu_{p,n}$ are said to be level-$n$
components of $\theta$ and $\nu$ respectively.

Throughout the rest of the paper, we use the probabilistic notations
introduced in \cite[Section 2.2]{Ho1}; readers are encouraged to
consult this reference for further details. In particular, we often
consider $\theta_{x,n}$ and $\nu_{p,n}$ as random measures in a
natural way.

\section{\label{sec:An-entropy-increase-result}An entropy increase result}

The purpose of this section is to prove the following theorem, which,
as explained in Section \ref{subsec:About-the-proof}, plays a key
role in the proof of Theorem \ref{thm:main thm}. In what follows,
we denote by $C^{1}(I,I)$ the set of $C^{1}$-maps from $I$ into
itself.
\begin{thm}
\label{thm:ent inc with poly}Suppose that $\dim\mu<1$. Then for
each $k\in\mathbb{Z}_{>0}$ and $0<\epsilon<1$ there exists $\rho=\rho(k,\epsilon)>0$
such that the following holds for all $n\ge N(k,\epsilon)\ge1$. Let
$\nu\in\mathcal{M}(\mathcal{P}_{k})$ be such that $\Vert p\Vert_{2}\le\epsilon^{-1}$
and $\left|p'(x)\right|\ge\epsilon$ for all $p\in\mathrm{supp}(\nu)$
and $x\in I$. Suppose also that $\frac{1}{n}H\left(\nu,\mathcal{D}_{n}\right)\ge\epsilon$,
and let $\psi\in C^{1}(I,I)$ be with $\epsilon\le\left|\psi'(x)\right|\le\epsilon^{-1}$
for all $x\in I$. Then $\frac{1}{n}H\left(\nu.\left(\psi\mu\right),\mathcal{D}_{n}\right)\ge\dim\mu+\rho$.
\end{thm}

In Sections \ref{subsec:Entropy-growth-under conv in R}, \ref{subsec:Uniform-entropy-dimension}
and \ref{subsec:add prep for ent inc}, we make the necessary preparations
for the proof of Theorem \ref{thm:ent inc with poly}, which is carried
out in Section \ref{subsec:Proof-of-the ent in result}.

\subsection{\label{subsec:Entropy-growth-under conv in R}Entropy growth under
convolution in $\mathbb{R}$}

The following theorem, which is the main ingredient in the proof of
Theorem \ref{thm:ent inc with poly}, follows directly from Hochman's
work \cite[Theorem 2.8]{Ho1}; for more details, see \cite[Section 4.1]{BHR}.
\begin{thm}
\label{thm:ent inc in R}For every $0<\epsilon<1$, $m\ge1$ and $0<\eta<\eta(\epsilon)$,
there exists $\delta=\delta(\epsilon,m,\eta)>0$, such that for all
$n\ge N(\epsilon,m,\eta)\ge1$ the following holds. Let $i\in\mathbb{Z}_{>0}$
and $\xi,\theta\in\mathcal{M}\left(\mathbb{R}\right)$ be such that,
\begin{enumerate}
\item $\mathrm{diam}\left(\mathrm{supp}(\xi)\right),\mathrm{diam}\left(\mathrm{supp}(\theta)\right)\le\epsilon^{-1}2^{-i}$;
\item $\frac{1}{n}H\left(\xi,\mathcal{D}_{i+n}\right)>\epsilon$;
\item $\mathbb{P}_{i\le j\le i+n}\left\{ \frac{1}{m}H\left(\theta_{x,j},\mathcal{D}_{j+m}\right)<1-\epsilon\right\} >1-\eta$.
\end{enumerate}
Then,
\[
\frac{1}{n}H\left(\xi*\theta,\mathcal{D}_{i+n}\right)\ge\frac{1}{n}H\left(\theta,\mathcal{D}_{i+n}\right)+\delta.
\]
\end{thm}

\subsection{\label{subsec:Uniform-entropy-dimension}Uniform entropy dimension}

The following proposition, which is required to apply Theorem \ref{thm:ent inc in R}
in the proof of Theorem \ref{thm:ent inc with poly}, is the main
result of this subsection. Roughly speaking, in the terminology of
\cite{Ho1}, it states that $C^{1}$-images of $\mu$ have uniform
entropy dimension $\dim\mu$.
\begin{prop}
\label{prop:uni ent dim}For each $0<\epsilon<1$, $m\ge M(\epsilon)\ge1$
and $n\ge N(\epsilon,m)\ge1$ the following holds. Let $\psi\in C^{1}(I)$
be with $\epsilon\le\left|\psi'(x)\right|\le\epsilon^{-1}$ for $x\in I$.
Then, setting $\theta:=\psi\mu$,
\[
\mathbb{P}_{1\le i\le n}\left(\left|\frac{1}{m}H\left(\theta_{x,i},\mathcal{D}_{i+m}\right)-\dim\mu\right|<\epsilon\right)>1-\epsilon.
\]
\end{prop}

The proof of the proposition requires some preparations. We begin
with the following lemma. In the language of \cite[Section 3.2]{BHR},
it says that $C^{1}$-images of $\mu$ are uniformly continuous across
scales.
\begin{lem}
\label{lem:uni cont ac scale}For every $0<\epsilon<1$ there exists
$\delta>0$ such that the following holds. Let $\psi\in C^{1}(I)$
be with $\epsilon\le\left|\psi'(x)\right|\le\epsilon^{-1}$ for $x\in I$.
Then,
\[
\psi\mu\left(B(x,\delta r)\right)\le\epsilon\cdot\psi\mu\left(B(x,r)\right)\text{ for each }x\in\mathbb{R}\text{ and }0<r<1.
\]
\end{lem}

\begin{proof}
Let $C>1$ be a large global constant depending only on $\Phi$, and
let $0<\epsilon,\eta,\delta<1$ and $m\in\mathbb{Z}_{>0}$ be with
\begin{equation}
C,\epsilon^{-1}\ll\eta^{-1}\ll m\ll\delta^{-1}.\label{eq:rel between prams in cont lemma}
\end{equation}
Since $\mu$ is nonatomic (see Section \ref{subsec:The-setup}) and
$\epsilon^{-1}\ll\eta^{-1}$, we may assume that
\begin{equation}
\mu(J)<\epsilon\text{ for each interval }J\subset\mathbb{R}\text{ of length at most }\eta.\label{eq:mu(J)<epsilon}
\end{equation}

Let $\psi\in C^{1}(I)$ be with $\epsilon\le\left|\psi'(y)\right|\le\epsilon^{-1}$
for $y\in I$, fix $x\in\mathbb{R}$ and $0<r<1$, and set $n:=\left\lfloor -\log r\right\rfloor $.
Let $u\in\Psi_{n+m}$ be given, where $\Psi_{n+m}$ is defined in
Section \ref{subsec:Symbolic-notations}. By (\ref{eq:der for u in Psi_n})
and the choice of $n$,
\begin{equation}
\epsilon C^{-1}2^{-m}r\le\left|\left(\psi\circ\varphi_{u}\right)'(y)\right|\le\epsilon^{-1}C2^{-m}r\text{ for }y\in I.\label{eq:lb ub on der in cont lem}
\end{equation}
Thus, $\left(\psi\circ\varphi_{u}\right)^{-1}\left(B(x,\delta r)\right)$
is an interval of length at most $2\delta\epsilon^{-1}C2^{m}$. From
this, (\ref{eq:rel between prams in cont lemma}) and (\ref{eq:mu(J)<epsilon}),
\begin{equation}
\psi\varphi_{u}\mu\left(B(x,\delta r)\right)<\epsilon\text{ for }u\in\Psi_{n+m}.\label{eq:ub on psi varph mu (B(x,delta r))}
\end{equation}

Next, let $u\in\Psi_{n+m}$ be with $\psi\varphi_{u}\mu\left(B(x,\delta r)\right)>0$.
Then there exists $z\in\psi\circ\varphi_{u}(I)$ with $|z-x|\le\delta r$.
Additionally, from (\ref{eq:lb ub on der in cont lem}) it follows
that $\mathrm{diam}\left(\psi\circ\varphi_{u}(I)\right)\le\epsilon^{-1}C2^{-m}r$.
Hence, for $y\in\psi\circ\varphi_{u}(I)$,
\[
|y-x|\le\epsilon^{-1}C2^{-m}r+\delta r<r,
\]
where the last inequality follows from (\ref{eq:rel between prams in cont lemma}).
This implies that $\psi\circ\varphi_{u}(I)\subset B(x,r)$, and so
\[
\psi\varphi_{u}\mu\left(B(x,r)\right)=1\text{ for }u\in\Psi_{n+m}\text{ with }\psi\varphi_{u}\mu\left(B(x,\delta r)\right)>0.
\]
Thus, from (\ref{eq:mu=00003DSum_=00007Bu in Psi_n=00007D}) and (\ref{eq:ub on psi varph mu (B(x,delta r))}),
\begin{multline*}
\psi\mu\left(B(x,\delta r)\right)\le\epsilon\sum_{u\in\Psi_{n+m}}p_{u}\cdot1_{\left\{ \psi\varphi_{u}\mu\left(B(x,\delta r)\right)>0\right\} }\\
\le\epsilon\sum_{u\in\Psi_{n+m}}p_{u}\cdot\psi\varphi_{u}\mu\left(B(x,r)\right)=\epsilon\cdot\psi\mu\left(B(x,r)\right),
\end{multline*}
which completes the proof of the lemma.
\end{proof}
Given $\emptyset\ne E\subset\mathbb{R}$ and $\delta>0$, write $E^{(\delta)}:=\left\{ x\in\mathbb{R}\::\:d\left(x,E\right)\le\delta\right\} $
for the closed $\delta$-neighbourhood of $E$. The following corollary
follows directly from Lemma \ref{lem:uni cont ac scale}.
\begin{cor}
\label{cor:from uni cont}For every $0<\epsilon<1$ there exists $\delta>0$
such that the following holds. Let $\psi\in C^{1}(I)$ be with $\epsilon\le\left|\psi'(x)\right|\le\epsilon^{-1}$
for $x\in I$. Then,
\[
\psi\mu\left(\cup_{D\in\mathcal{D}_{n}^{1}}\left(\partial D\right)^{\left(2^{-n}\delta\right)}\right)<\epsilon\text{ for each }n\ge0.
\]
\end{cor}

Most of the proof of Proposition \ref{prop:uni ent dim} is contained
in the proof of the following proposition.
\begin{prop}
\label{prop:bef uni ent dim prop}For each $0<\epsilon<1$, $m\ge M(\epsilon)\ge1$
and $n\ge1$ the following holds. Let $\psi\in C^{1}(I)$ be with
$\epsilon\le\left|\psi'(x)\right|\le\epsilon^{-1}$ for $x\in I$.
Then, setting $\theta:=\psi\mu$,
\[
\mathbb{P}\left(\frac{1}{m}H\left(\theta_{x,n},\mathcal{D}_{n+m}\right)>\dim\mu-\epsilon\right)>1-\epsilon.
\]
\end{prop}

\begin{proof}
Let $C>1$ be a large global constant depending only on $\Phi$, and
let $0<\epsilon,\delta<1$ and $k,m,n\in\mathbb{Z}_{>0}$ be with
\[
C,\epsilon^{-1}\ll\delta^{-1}\ll k\ll m.
\]
Let $\psi\in C^{1}(I)$ be such that $\epsilon\le\left|\psi'(x)\right|\le\epsilon^{-1}$
for $x\in I$.

Set $B:=\cup_{D\in\mathcal{D}_{n}^{1}}\left(\partial D\right)^{\left(2^{-n}\delta\right)}$.
By Corollary \ref{cor:from uni cont} and since $\epsilon^{-1}\ll\delta^{-1}$,
we have $\psi\mu(B)<\epsilon^{2}$. Let $\mathcal{U}$ be the set
of words $u\in\Psi_{n+k}$ such that $\psi\circ\varphi_{u}(I)\subset D$
for some $D\in\mathcal{D}_{n}^{1}$. By (\ref{eq:der for u in Psi_n}),
\[
\mathrm{diam}\left(\psi\circ\varphi_{u}(I)\right)\le C\epsilon^{-1}2^{-n-k}\text{ for }u\in\Psi_{n+k}.
\]
From this and since $C,\epsilon^{-1},\delta^{-1}\ll k$, it follows
that $\psi\circ\varphi_{u}(I)\subset B$ for $u\in\Psi_{n+k}\setminus\mathcal{U}$.
Thus, by (\ref{eq:mu=00003DSum_=00007Bu in Psi_n=00007D}),
\[
\sum_{u\in\Psi_{n+k}\setminus\mathcal{U}}p_{u}\le\sum_{u\in\Psi_{n+k}}p_{u}\cdot\psi\varphi_{u}\mu(B)=\psi\mu(B)<\epsilon^{2},
\]
which implies,
\begin{equation}
\epsilon^{2}>\sum_{D\in\mathcal{D}_{n}^{1}}\psi\mu(D)\sum_{u\in\Psi_{n+k}\setminus\mathcal{U}}p_{u}\cdot\frac{\psi\varphi_{u}\mu(D)}{\psi\mu(D)}.\label{eq:eta^2>}
\end{equation}

Let $\mathcal{E}$ be the set of $D\in\mathcal{D}_{n}^{1}$ such that
$\psi\mu(D)>0$ and
\[
\frac{1}{\psi\mu(D)}\sum_{u\in\Psi_{n+k}\setminus\mathcal{U}}p_{u}\cdot\psi\varphi_{u}\mu(D)<\epsilon.
\]
By (\ref{eq:eta^2>}), we have $\psi\mu\left(\cup\mathcal{E}\right)>1-\epsilon$.

Fix $D\in\mathcal{E}$ and let $\mathcal{U}_{D}$ be the set of $u\in\Psi_{n+k}$
such that $\psi\circ\varphi_{u}(I)\subset D$. By (\ref{eq:mu=00003DSum_=00007Bu in Psi_n=00007D}),
\begin{equation}
\left(\psi\mu\right)_{D}=\frac{1}{\psi\mu(D)}\sum_{u\in\Psi_{n+k}}p_{u}\cdot\psi\varphi_{u}\mu(D)\cdot\left(\psi\varphi_{u}\mu\right)_{D}.\label{eq:(psi mu)_D =00003D}
\end{equation}
Additionally, from $D\in\mathcal{E}$ and since $\psi\varphi_{u}\mu(D)=0$
for $u\in\mathcal{U}\setminus\mathcal{U}_{D}$,
\begin{equation}
\frac{1}{\psi\mu(D)}\sum_{u\in\mathcal{U}_{D}}p_{u}\cdot\psi\varphi_{u}\mu(D)>1-\epsilon.\label{eq:weights in U_D >}
\end{equation}

Given $u\in\mathcal{U}_{D}$, we have $\left(\psi\varphi_{u}\mu\right)_{D}=\psi\varphi_{u}\mu$.
Hence,
\begin{equation}
\frac{1}{m}H\left(\left(\psi\varphi_{u}\mu\right)_{D},\mathcal{D}_{n+m}\right)=\frac{1}{m}H\left(S_{2^{n}}\psi\varphi_{u}\mu,\mathcal{D}_{m}\right),\label{eq:ent of (psi varph mu)_D}
\end{equation}
where recall that $S_{2^{n}}(y)=2^{n}y$ for $y\in\mathbb{R}$. Since
$u\in\Psi_{n+k}$ and by (\ref{eq:der for u in Psi_n}),
\[
\epsilon C^{-1}2^{-k}\le\left|\left(S_{2^{n}}\psi\varphi_{u}\right)'(x)\right|\le\epsilon^{-1}C2^{-k}\text{ for }x\in I.
\]
From this, (\ref{eq:ent of (psi varph mu)_D}), Lemma \ref{lem:dyad ent =000026 lip func},
(\ref{eq:ent dim =00003D exact dim}), and $\epsilon^{-1},C,k,\ll m$,
\[
\frac{1}{m}H\left(\left(\psi\varphi_{u}\mu\right)_{D},\mathcal{D}_{n+m}\right)\ge\dim\mu-\epsilon\text{ for }u\in\mathcal{U}_{D}.
\]
Thus, by concavity of entropy, (\ref{eq:(psi mu)_D =00003D}), and
(\ref{eq:weights in U_D >}),
\[
\frac{1}{m}H\left(\left(\psi\mu\right)_{D},\mathcal{D}_{n+m}\right)\ge(1-\epsilon)\left(\dim\mu-\epsilon\right)>\dim\mu-2\epsilon\text{ for }D\in\mathcal{E}.
\]
Since $\psi\mu\left(\cup\mathcal{E}\right)>1-\epsilon$, this completes
the proof of the proposition.
\end{proof}
\begin{proof}[Proof of Proposition \ref{prop:uni ent dim}]
Let $0<\epsilon<1$ and $m,n\in\mathbb{Z}_{>0}$ be with $\epsilon^{-1}\ll m\ll n$.
Let $\psi\in C^{1}(I)$ be with $\epsilon\le\left|\psi'(x)\right|\le\epsilon^{-1}$
for $x\in I$, and set $\theta:=\psi\mu$. By Proposition \ref{prop:bef uni ent dim prop},
\[
\mathbb{P}_{1\le i\le n}\left(\frac{1}{m}H\left(\theta_{x,i},\mathcal{D}_{i+m}\right)>\dim\mu-\epsilon\right)>1-\epsilon.
\]
Moreover, by \cite[Lemma 3.4]{Ho1}, Lemma \ref{lem:dyad ent =000026 lip func}, and (\ref{eq:ent dim =00003D exact dim}),
\[
\mathbb{E}_{1\le i\le n}\left(\frac{1}{m}H\left(\theta_{x,i},\mathcal{D}_{i+m}\right)\right)=\dim\mu+O(\epsilon).
\]
These two facts complete the proof of the proposition (by starting
from a smaller $\epsilon$).
\end{proof}
In the proof of Theorem \ref{thm:ent inc with poly}, we shall need
the following corollary, which follows directly from Proposition \ref{prop:uni ent dim}
and \cite[Lemma 2.7]{Ho}. Recall the notation $\mathcal{N}_{n}$
and $\lambda_{n}$ introduced in Section \ref{subsec:Basic-notations}.
\begin{cor}
\label{cor:uni ent dim comp of comp}For each $0<\epsilon<1$, $m\ge1$,
$l\ge L(\epsilon)\ge1$ and $n\ge N(\epsilon,m,l)\ge1$ the following
holds. Let $\psi\in C^{1}(I)$ be with $\epsilon\le\left|\psi'(x)\right|\le\epsilon^{-1}$
for $x\in I$, and set $\theta:=\psi\mu$. Then $\lambda_{n}\times\theta(E)>1-\epsilon$,
where $E$ is the set of all $(i,x)\in\mathcal{N}_{n}\times\mathbb{R}$
such that
\[
\mathbb{P}_{i\le j\le i+m}\left\{ \left|\frac{1}{l}H\left(\left(\theta_{x,i}\right)_{y,j},\mathcal{D}_{j+l}\right)-\dim\mu\right|<\epsilon\right\} >1-\epsilon.
\]
\end{cor}

\subsection{\label{subsec:add prep for ent inc}Additional preparations for the
proof of Theorem \ref{thm:ent inc with poly}}

The proof of the following lemma is similar to the proof of \cite[Lemma 6.9]{HR}
and is therefore omitted.
\begin{lem}
\label{lem:step1 in ent inc pf}Let $k\in\mathbb{Z}_{>0}$, $R>1$,
$\nu\in\mathcal{M}(\mathcal{P}_{k})$ and $\theta\in\mathcal{M}(I)$
be given. Suppose that $\Vert p\Vert_{2}\le R$ for $p\in\mathrm{supp}(\nu)$.
Then for every $n\ge m\ge1,$
\[
\frac{1}{n}H\left(\nu.\theta,\mathcal{D}_{n}\right)\ge\mathbb{E}_{1\le i\le n}\left(\frac{1}{m}H\left(\nu_{p,i}.\theta_{x,i},\mathcal{D}_{i+m}\right)\right)-O_{k,R}\left(\frac{1}{m}+\frac{m}{n}\right).
\]
\end{lem}

Writing $F$ for the map taking $(p,x)\in\mathcal{P}_{k}\times I$
to $p(x)$, the differential of $F$ at a point $(p,x)$ is given
by $dF_{(p,x)}(q,y)=q(x)+p'(x)y$. Using this, the proof of the following
linearization lemma is almost identical to the proof of \cite[Lemma 4.2]{BHR}
and is therefore omitted.
\begin{lem}
\label{lem:linearization}For every $\epsilon>0$, $k\in\mathbb{Z}_{>0}$,
$R>1$, $m\ge M(\epsilon)\ge1$, and $0<\delta<\delta(\epsilon,k,R,m)$,
the following holds. Let $p\in\mathcal{P}_{k}$, $x\in I$, $\nu\in\mathcal{M}(\mathcal{P}_{k})$
and $\theta\in\mathcal{M}(I)$ be such that $\Vert p\Vert_{2}\le R$,
$\Vert q-p\Vert_{2}\le\delta$ for $q\in\mathrm{supp}(\nu)$, and
$|y-x|\le\delta$ for $y\in\mathrm{supp}(\theta)$. Then,
\[
\left|\frac{1}{m}H\left(\nu.\theta,\mathcal{D}_{m-\log\delta}\right)-\frac{1}{m}H\left(\left(\nu.x\right)*\left(S_{p'(x)}\theta\right),\mathcal{D}_{m-\log\delta}\right)\right|<\epsilon.
\]
\end{lem}

The following lemma, which resembles \cite[Lemma 4.5]{BHR}, will
allow us to use the assumption $\frac{1}{n}H\left(\nu,\mathcal{D}_{n}^{\mathcal{P}_{k}}\right)\ge\epsilon$
in the proof of Theorem \ref{thm:ent inc with poly}.
\begin{lem}
\label{lem:ent of nu imp ent of push of comp of nu}For every $0<\epsilon<1$
and $k\in\mathbb{Z}_{>0}$ there exists $\epsilon_{0}=\epsilon_{0}(\epsilon,k)>0$
such that for all $m\ge M(\epsilon,k)\ge1$ and $n\ge N(\epsilon,k,m)\ge1$
the following holds. Let $\nu\in\mathcal{M}(\mathcal{P}_{k})$ be
such that $\Vert p\Vert_{2}\le\epsilon^{-1}$ for $p\in\mathrm{supp}(\nu)$
and $\frac{1}{n}H\left(\nu,\mathcal{D}_{n}\right)\ge\epsilon$. Let
$\psi\in C^{1}(I,I)$ be with $\epsilon\le\left|\psi'(x)\right|\le\epsilon^{-1}$
for $x\in I$, and set $\theta:=\psi\mu$. Then,
\begin{equation}
\int\mathbb{P}_{1\le i\le n}\left\{ \frac{1}{m}H\left(\left(\nu_{p,i}\right).x,\mathcal{D}_{i+m}\right)>\epsilon_{0}\right\} \:d\theta(x)>\epsilon_{0}.\label{eq:ent of nu.x nontrivial}
\end{equation}
\end{lem}

\begin{proof}
Let $0<\epsilon,\delta<1$, $k,m,n\in\mathbb{Z}_{>0}$ and $C>1$
be with
\[
\epsilon^{-1},k\ll\delta^{-1}\ll C\ll m\ll n.
\]
Let $\nu\in\mathcal{M}(\mathcal{P}_{k})$ be such that $\Vert p\Vert_{2}\le\epsilon^{-1}$
for $p\in\mathrm{supp}(\nu)$ and $\frac{1}{n}H\left(\nu,\mathcal{D}_{n}\right)\ge\epsilon$.
Let $\psi\in C^{1}(I,I)$ be with $\epsilon\le\left|\psi'(x)\right|\le\epsilon^{-1}$
for $x\in I$, and set $\theta:=\psi\mu$.

Note that the norm $\Vert\cdot\Vert_{2}$ on $\mathcal{P}_{k}$ and
the partitions $\mathcal{D}_{l}^{\mathcal{P}_{k}}$ are defined by
identifying $\mathcal{P}_{k}$ with the $(k+1)$-dimensional Euclidean
space $\mathbb{R}^{k+1}$. Thus, by \cite[Lemma 3.4]{Ho1}, since
$\Vert p\Vert_{2}\le\epsilon^{-1}$ for $p\in\mathrm{supp}(\nu)$,
and from $\epsilon^{-1},k,m\ll n$,
\[
\mathbb{E}_{1\le i\le n}\left(\frac{1}{m}H\left(\nu_{p,i},\mathcal{D}_{i+m}\right)\right)\ge\frac{1}{n}H\left(\nu,\mathcal{D}_{n}\right)-\epsilon/2\ge\epsilon/2.
\]
From this, and since $\frac{1}{m}H\left(\nu_{D},\mathcal{D}_{i+m}\right)\le k+1$
for all $i\ge0$ and $D\in\mathcal{D}_{i}^{\mathcal{P}_{k}}$ with
$\nu(D)>0$,
\begin{equation}
\mathbb{P}_{1\le i\le n}\left\{ \frac{1}{m}H\left(\nu_{p,i},\mathcal{D}_{i+m}\right)\ge\frac{\epsilon}{4}\right\} \ge\frac{\epsilon}{4(k+1)}.\label{eq:comp of nu have non neg ent with no neg prob}
\end{equation}

Write $\theta^{\times(k+1)}$ for the $(k+1)$-fold product of $\theta$
with itself. Let $E$ be the set of all $(x_{0},...,x_{k})\in I^{k+1}$
such that $|x_{j}-x_{l}|\ge\delta$ for all $0\le j<l\le k$. Since
$\mu$ is nonatomic and $\epsilon^{-1},k\ll\delta^{-1}$, we may assume
that $\theta\left(B(x,\delta)\right)<1/(2k)$ for all $x\in I$. This
easily implies that $\theta^{\times(k+1)}(E)\ge2^{-k}$.

Given $(x_{0},...,x_{k})=\overline{x}\in I^{k+1}$, let $L_{\overline{x}}:\mathcal{P}_{k}\rightarrow\mathbb{R}^{k+1}$
be the linear operator with $L_{\overline{x}}(p)=\left(p(x_{0}),...,p(x_{k})\right)$
for $p\in\mathcal{P}_{k}$. Note that when $x_{j}\ne x_{l}$ for $0\le j<l\le k$,
it holds that $L_{\overline{x}}$ is an isomorphism of vector spaces.
By compactness and since $\delta^{-1}\ll C$, this implies that
\begin{equation}
C^{-1}\Vert p\Vert_{2}\le\Vert L_{\overline{x}}(p)\Vert_{2}\le C\Vert p\Vert_{2}\:\text{ for all }\overline{x}\in E\text{ and }p\in\mathcal{P}_{k},\label{eq:bd for operators L}
\end{equation}
where we write $\Vert\cdot\Vert_{2}$ also for the $L^{2}$-norm on
$\mathbb{R}^{k+1}$.

Let $i\ge0$ and $D\in\mathcal{D}_{i}^{\mathcal{P}_{k}}$ be with
$\nu(D)>0$ and $\frac{1}{m}H\left(\nu_{D},\mathcal{D}_{i+m}\right)\ge\frac{\epsilon}{4}$.
By (\ref{eq:bd for operators L}), from Lemma \ref{lem:dyad ent =000026 lip func},
from $\epsilon^{-1},k,C\ll m$, and since the norm and dyadic partitions
of $\mathcal{P}_{k}$ are defined by identifying it with $\mathbb{R}^{k+1}$,
\[
\frac{1}{m}H\left(L_{\overline{x}}\nu_{D},\mathcal{D}_{i+m}\right)\ge\frac{1}{m}H\left(\nu_{D},\mathcal{D}_{i+m}\right)-\epsilon/8\ge\epsilon/8\:\text{ for }\overline{x}\in E.
\]
Together with $\theta^{\times(k+1)}(E)\ge2^{-k}$, this gives
\begin{equation}
\int\frac{1}{m}H\left(L_{\overline{x}}\nu_{D},\mathcal{D}_{i+m}\right)\:d\theta^{\times(k+1)}(\overline{x})\ge2^{-k-3}\epsilon.\label{eq:integral >=00003D2^-k-3 epsilon}
\end{equation}

For $0\le j\le k$, let $\pi_{j}:\mathbb{R}^{k+1}\rightarrow\mathbb{R}$
be the projection onto the $j$th coordinate of $\mathbb{R}^{k+1}$.
Given $(x_{0},...,x_{k})=\overline{x}\in I^{k+1}$, note that $\pi_{j}L_{\overline{x}}\nu_{D}=\nu_{D}.x_{j}$
for $0\le j\le k$. Hence, by the conditional entropy formula,
\[
H\left(L_{\overline{x}}\nu_{D},\mathcal{D}_{i+m}\right)\le\sum_{j=0}^{k}H\left(\pi_{j}L_{\overline{x}}\nu_{D},\mathcal{D}_{i+m}\right)=\sum_{j=0}^{k}H\left(\nu_{D}.x_{j},\mathcal{D}_{i+m}\right).
\]
Together with (\ref{eq:integral >=00003D2^-k-3 epsilon}), this gives
\begin{eqnarray*}
2^{-k-3}\epsilon & \le & \sum_{j=0}^{k}\int\frac{1}{m}H\left(\nu_{D}.x_{j},\mathcal{D}_{i+m}\right)\:d\theta^{\times(k+1)}(x_{0},...,x_{k})\\
 & = & (k+1)\int\frac{1}{m}H\left(\nu_{D}.x,\mathcal{D}_{i+m}\right)\:d\theta(x).
\end{eqnarray*}

We have thus shown that for all $i\ge0$ and $D\in\mathcal{D}_{i}^{\mathcal{P}_{k}}$
with $\nu(D)>0$ and $\frac{1}{m}H\left(\nu_{D},\mathcal{D}_{i+m}\right)\ge\frac{\epsilon}{4}$,
\[
\int\frac{1}{m}H\left(\nu_{D}.x,\mathcal{D}_{i+m}\right)\:d\theta(x)\ge\frac{\epsilon}{(k+1)2^{k+3}}.
\]
Together with (\ref{eq:comp of nu have non neg ent with no neg prob}),
this implies
\begin{equation}
\int\mathbb{E}_{1\le i\le n}\left\{ \frac{1}{m}H\left(\left(\nu_{p,i}\right).x,\mathcal{D}_{i+m}\right)\right\} \:d\theta(x)\ge\frac{\epsilon^{2}}{(k+1)^{2}2^{k+5}}.\label{eq:bd of exp of nu.x}
\end{equation}

Given $i\ge0$, $D\in\mathcal{D}_{i}^{\mathcal{P}_{k}}$ with $\nu(D)>0$,
$p_{1},p_{2}\in D$, and $x\in I$, we have $\left|p_{1}(x)-p_{2}(x)\right|\le(k+1)2^{-i}$.
Hence $\mathrm{diam}\left(\left(\nu_{D}\right).x\right)\le(k+1)2^{-i}$,
which implies
\[
\frac{1}{m}H\left(\left(\nu_{D}\right).x,\mathcal{D}_{i+m}\right)\le1+\frac{1}{m}\log(k+2)\le2.
\]
Setting $\epsilon_{0}:=\epsilon^{2}(k+1)^{-2}2^{-k-7}$, together
with (\ref{eq:bd of exp of nu.x}) this gives (\ref{eq:ent of nu.x nontrivial}),
which completes the proof of the lemma.
\end{proof}

\subsection{\label{subsec:Proof-of-the ent in result}Proof of the entropy increase
result}
\begin{proof}[Proof of Theorem \ref{thm:ent inc with poly}]
Suppose that $\dim\mu<1$, and let $k,l,m,n\in\mathbb{Z}_{>0}$ and
$0<\epsilon,\epsilon_{0},\eta,\rho,\delta<1$ be with,
\begin{equation}
\left(1-\dim\mu\right)^{-1},k,\epsilon^{-1}\ll\epsilon_{0}^{-1}\ll\eta^{-1}\ll l\ll\rho^{-1}\ll\delta^{-1}\ll m\ll n.\label{eq:rel between params in ent inc result}
\end{equation}
Let $\nu\in\mathcal{M}(\mathcal{P}_{k})$ be such that $\Vert p\Vert_{2}\le\epsilon^{-1}$
and $\left|p'(x)\right|\ge\epsilon$ for $p\in\mathrm{supp}(\nu)$
and $x\in I$, and $\frac{1}{n}H\left(\nu,\mathcal{D}_{n}\right)\ge\epsilon$.
Let $\psi\in C^{1}(I,I)$ be with $\epsilon\le\left|\psi'(x)\right|\le\epsilon^{-1}$
for $x\in I$, and set $\theta:=\psi\mu$.

By Lemma \ref{lem:step1 in ent inc pf},
\[
\frac{1}{n}H\left(\nu.\theta,\mathcal{D}_{n}\right)+\delta\ge\mathbb{E}_{1\le i\le n}\left(\frac{1}{m}H\left(\nu_{p,i}.\theta_{x,i},\mathcal{D}_{i+m}\right)\right).
\]
Hence, by Lemma \ref{lem:linearization},
\[
\frac{1}{n}H\left(\nu.\theta,\mathcal{D}_{n}\right)+2\delta\ge\mathbb{E}_{1\le i\le n}\left(\frac{1}{m}H\left(\left(\nu_{p,i}.x\right)*\left(S_{p'(x)}\theta_{x,i}\right),\mathcal{D}_{i+m}\right)\right).
\]
Thus, since $\left|p'(x)\right|\ge\epsilon$ for $p\in\mathrm{supp}(\nu)$
and $x\in I$,
\begin{equation}
\frac{1}{n}H\left(\nu.\theta,\mathcal{D}_{n}\right)+3\delta\ge\mathbb{E}_{1\le i\le n}\left(\frac{1}{m}H\left(\left(S_{p'(x)^{-1}}\nu_{p,i}\right).x*\theta_{x,i},\mathcal{D}_{i+m}\right)\right).\label{eq:by linearization ect}
\end{equation}

Write $\Gamma:=\lambda_{n}\times\theta\times\nu$, where $\lambda_{n}$
is defined in Section \ref{subsec:Basic-notations}. Let $E_{1}$
be the set of all $(i,x,p)\in\mathcal{N}_{n}\times I\times\mathrm{supp}(\nu)$
such that $\frac{1}{m}H\left(\theta_{x,i},\mathcal{D}_{i+m}\right)\ge\dim\mu-\delta$.
By Proposition \ref{prop:bef uni ent dim prop}, we may assume that
$\Gamma(E_{1})>1-\delta$. Also, by \cite[Corollary 4.10]{Ho1},
\begin{equation}
\frac{1}{m}H\left(\left(S_{p'(x)^{-1}}\nu_{p,i}\right).x*\theta_{x,i},\mathcal{D}_{i+m}\right)\ge\dim\mu-2\delta\text{ for }(i,x,p)\in E_{1}.\label{eq:lb on E_1}
\end{equation}

Let $E_{2}$ be the set of all $(i,x,p)\in E_{1}$ such that
\begin{equation}
\mathbb{P}_{i\le j\le i+m}\left\{ \frac{1}{l}H\left(\left(\theta_{x,i}\right)_{y,j},\mathcal{D}_{j+l}\right)<\frac{1+\dim\mu}{2}\right\} >1-\eta,\label{eq:first prop of E_2}
\end{equation}
and
\begin{equation}
\frac{1}{m}H\left(\left(\nu_{p,i}\right).x,\mathcal{D}_{i+m}\right)>\epsilon_{0}.\label{eq:second prop of E_2}
\end{equation}
By Corollary \ref{cor:uni ent dim comp of comp} and Lemma \ref{lem:ent of nu imp ent of push of comp of nu},
we may assume that $\Gamma(E_{2})>\epsilon_{0}$.

Let $(i,x,p)\in E_{2}$ and set $\xi:=\left(S_{p'(x)^{-1}}\nu_{p,i}\right).x$.
We next want to apply Theorem \ref{thm:ent inc in R} in order to
obtain entropy increase for the convolution $\xi*\theta_{x,i}$. Let
$q_{1},q_{2}\in\mathrm{supp}(\nu_{p,i})$ be given. Since $x\in I$,
$\left|p'(x)\right|\ge\epsilon$, and $q_{1},q_{2}$ belong
to the same atom of $\mathcal{D}_{i}^{\mathcal{P}_{k}}$,
\[
\left|p'(x)\right|^{-1}\cdot\left|q_{1}(x)-q_{2}(x)\right|\le\epsilon^{-1}(k+1)2^{-i}.
\]
From this and since $\theta_{x,i}$ is supported on a single atom
of $\mathcal{D}_{i}^{1}$,
\begin{equation}
\mathrm{diam}\left(\mathrm{supp}\left(\xi\right)\right),\mathrm{diam}\left(\mathrm{supp}\left(\theta_{x,i}\right)\right)=O_{k,\epsilon}\left(2^{-i}\right).\label{eq:supports are big O}
\end{equation}

We have $\Vert p\Vert_{2}\le\epsilon^{-1}$, which implies $\left|p'(x)\right|=O_{k,\epsilon}(1)$.
Thus, from (\ref{eq:second prop of E_2}) and since $k,\epsilon^{-1},\epsilon_{0}^{-1}\ll m$,
\[
\frac{1}{m}H\left(\xi,\mathcal{D}_{i+m}\right)>\epsilon_{0}/2.
\]
From this, (\ref{eq:rel between params in ent inc result}), (\ref{eq:first prop of E_2}),
(\ref{eq:supports are big O}), and Theorem \ref{thm:ent inc in R},
\[
\frac{1}{m}H\left(\xi*\theta_{x,i},\mathcal{D}_{i+m}\right)\ge\frac{1}{m}H\left(\theta_{x,i},\mathcal{D}_{i+m}\right)+\rho.
\]
Hence, since $E_{2}\subset E_{1}$, we have thus proven that
\begin{equation}
\frac{1}{m}H\left(\left(S_{p'(x)^{-1}}\nu_{p,i}\right).x*\theta_{x,i},\mathcal{D}_{i+m}\right)\ge\dim\mu+\rho-\delta\text{ for }(i,x,p)\in E_{2}.\label{eq:lb on E_2}
\end{equation}

Now, from (\ref{eq:by linearization ect}), (\ref{eq:lb on E_1})
and (\ref{eq:lb on E_2}),
\[
\frac{1}{n}H\left(\nu.\theta,\mathcal{D}_{n}\right)+3\delta\ge\Gamma\left(E_{1}\setminus E_{2}\right)\left(\dim\mu-2\delta\right)+\Gamma\left(E_{2}\right)\left(\dim\mu+\rho-\delta\right).
\]
Thus, since $\Gamma(E_{1})>1-\delta$ and $\Gamma(E_{2})>\epsilon_{0}$,
\[
\frac{1}{n}H\left(\nu.\theta,\mathcal{D}_{n}\right)\ge\dim\mu+\epsilon_{0}\rho-O(\delta).
\]
By (\ref{eq:rel between params in ent inc result}), this completes
the proof of the theorem.
\end{proof}

\section{\label{sec:Proof-of-the-main-result}Proof of the main result}

We shall need the following lemma for the proof of Theorem \ref{thm:main thm}.
Recall from Section \ref{subsec:Symbolic-notations} that $\Pi:\Lambda^{\mathbb{N}}\rightarrow I$
denotes the coding map associated to $\Phi$, and that $\{\beta_{\omega}\}_{\omega\in\Lambda^{\mathbb{N}}}$
denotes the disintegration of $\beta:=p^{\mathbb{N}}$ with respect
to $\Pi^{-1}\mathcal{B}_{\mathbb{R}}$. Also, recall that for each
$n\ge1$, we defined $\Pi_{n}:\Lambda^{\mathbb{N}}\rightarrow\mathcal{A}(I)$
by $\Pi_{n}(\omega)=\varphi_{\omega|_{n}}$ for $\omega\in\Lambda^{\mathbb{N}}$.
Finally, recall from Section \ref{subsec:The-setup} that $\chi$
denotes the Lyapunov exponent associated to $\Phi$ and $p$.
\begin{lem}
\label{lem:0-ent lemma}For $\beta$-a.e. $\omega$,
\[
\underset{n\rightarrow\infty}{\lim}\frac{1}{n}H\left(\left(\Pi_{n}\beta_{\omega}\right).\mu,\mathcal{D}_{\chi n}\right)=0.
\]
\end{lem}

\begin{proof}
From (\ref{eq:chi as a.s.  limit}) and by basic basic properties
of disintegrations (see \cite[Section 5]{EiWa}), for $\beta$-a.e.
$\omega$ we have $\beta_{\omega}\left(\Pi^{-1}\left(\Pi\omega\right)\right)=1$
and
\begin{equation}
\underset{l\rightarrow\infty}{\lim}\frac{1}{l}\log\left|\varphi_{\eta|_{l}}'(0)\right|=-\chi\text{ for }\beta_{\omega}\text{-a.e. }\eta\in \Lambda^{\mathbb{N}}.\label{eq:beta_omega a.e.  eta}
\end{equation}
Fix $\omega\in\Lambda^{\mathbb{N}}$ for which these properties hold.

Let $C>1$ be a large global constant depending only on $\Phi$, let
$0<\delta<1$ and $n\in\mathbb{Z}_{>0}$ be with $C,\delta^{-1}\ll n$,
and let $\mathcal{U}$ be the set of words $u\in\Lambda^{n}$ such
that $\beta_{\omega}\left([u]\right)>0$ and $\left|\varphi_{u}'(0)\right|\le2^{n(\delta-\chi)}$.
By (\ref{eq:beta_omega a.e.  eta}) and $\delta^{-1}\ll n$, we may
assume that $\beta_{\omega}\left(\left[\mathcal{U}\right]\right)>1-\delta$,
where recall that we write $\left[\mathcal{U}\right]$ in place of
$\cup_{u\in\mathcal{U}}[u]$.

Let $u\in\mathcal{U}$ be given. From $\beta_{\omega}\left([u]\right)>0$
and $\beta_{\omega}\left(\Pi^{-1}\left(\Pi\omega\right)\right)=1$,
it follows that there exists $\eta\in[u]\cap\Pi^{-1}\left(\Pi\omega\right)$.
Thus, for each $x\in I$,
\[
\left|\varphi_{u}(x)-\Pi(\omega)\right|=\left|\varphi_{u}(x)-\varphi_{u}\left(\Pi(\sigma^{n}(\eta))\right)\right|\le C2^{n(\delta-\chi)},
\]
where the last inequality follows from the mean value theorem and
bounded distortion. This implies that $\mathrm{supp}\left(\varphi_{u}\mu\right)\subset B\left(\Pi(\omega),C2^{n(\delta-\chi)}\right)$
for all $u\in\mathcal{U}$, and so
\[
\mathrm{diam}\left(\mathrm{supp}\left(\left(\Pi_{n}\left(\beta_{\omega}\right)_{\left[\mathcal{U}\right]}\right).\mu\right)\right)\le C2^{1+n(\delta-\chi)}.
\]
Hence, since $C,\delta^{-1}\ll n$,
\begin{equation}
\frac{1}{n}H\left(\left(\Pi_{n}\left(\beta_{\omega}\right)_{\left[\mathcal{U}\right]}\right).\mu,\mathcal{D}_{\chi n}\right)\le2\delta.\label{eq:ub cond on =00005BU=00005D}
\end{equation}

Set $\mathcal{U}^{c}:=\Lambda^{n}\setminus\mathcal{U}$, and note
that $\beta_{\omega}\left(\left[\mathcal{U}^{c}\right]\right)<\delta$.
Thus, since $\left(\Pi_{n}\left(\beta_{\omega}\right)_{\left[\mathcal{U}^{c}\right]}\right).\mu$
is supported on $I$,
\[
\beta_{\omega}\left(\left[\mathcal{U}^{c}\right]\right)\frac{1}{n}H\left(\left(\Pi_{n}\left(\beta_{\omega}\right)_{\left[\mathcal{U}^{c}\right]}\right).\mu,\mathcal{D}_{\chi n}\right)<\chi\delta.
\]
From this, from (\ref{eq:ub cond on =00005BU=00005D}), by the convexity
bound for entropy (see \cite[Lemma 3.1]{Ho1}), and since $\delta^{-1}\ll n$,
\[
\frac{1}{n}H\left(\left(\Pi_{n}\beta_{\omega}\right).\mu,\mathcal{D}_{\chi n}\right)\le(3+\chi)\delta,
\]
which completes the proof of the lemma.
\end{proof}
We can now begin the proof of our main result.
\begin{proof}[Proof of Theorem \ref{thm:main thm}]
Suppose that the assumptions made in the theorem are all satisfied,
and assume by contradiction that $\dim\mu<\min\left\{ 1,H(p)/\chi\right\} $.

Recall that for $l\ge1$, we denote by $\mathcal{C}_{l}$ the partition
of $\Lambda^{\mathbb{N}}$ into level-$l$ cylinder sets. Set $\Delta':=H\left(\beta,\mathcal{C}_{1}\mid\Pi^{-1}\mathcal{B}_{\mathbb{R}}\right)$,
where the right-hand side denotes the conditional entropy of $\mathcal{C}_{1}$
given $\Pi^{-1}\mathcal{B}_{\mathbb{R}}$ with respect to $\beta$.
By \cite[Theorem 2.8]{FH-dimension}, we have $\dim\mu=\left(H(p)-\Delta'\right)/\chi$.
Thus, from $\dim\mu<H(p)/\chi$, it follows that $\Delta'>0$. Moreover,
by \cite[Proposition 4.10]{FH-dimension},
\begin{equation}
\underset{l\rightarrow\infty}{\lim}\frac{1}{l}H\left(\beta_{\omega},\mathcal{C}_{l}\right)=\Delta'\text{ for }\beta\text{-a.e. }\omega.\label{eq:ent of slices}
\end{equation}

Since $\Phi$ is exponentially separated, there exists $0<c<1$ as
in Definition \ref{def:exp sep}. Set $\Delta:=\min\left\{ \Delta',1\right\} /2$,
and let $C>1$ be a large global constant depending only on $\Phi$.
Let $0<\rho,\delta<1$ and $M,k,n\in\mathbb{Z}_{>0}$ be such that
\begin{equation}
\Delta^{-1},c^{-1},C\ll M\ll k\ll\rho^{-1}\ll\delta^{-1}\ll n.\label{eq:rel between params in main pf}
\end{equation}
Set $n':=\left\lfloor n\Delta/\left(2\log|\Lambda|\right)\right\rfloor $.
By exponential separation and the choice of $c$, we may clearly assume
that
\begin{equation}
\Vert\varphi_{u_{1}}-\varphi_{u_{2}}\Vert_{I}\ge c^{n+n'}\text{ for all distinct }u_{1},u_{2}\in\Lambda^{n+n'}.\label{eq:by exp sep}
\end{equation}
By (\ref{eq:ent dim =00003D exact dim}) and $\delta^{-1}\ll n$,
we may also assume that
\begin{equation}
\dim\mu+\delta\ge\frac{1}{Mn}H\left(\mu,\mathcal{D}_{Mn+\chi(n+n')}\mid\mathcal{D}_{\chi(n+n')}\right).\label{eq:lb by exact dim}
\end{equation}

By (\ref{eq:def rel of mu}) and since $\beta=\int\beta_{\omega}\:d\beta(\omega)$,
\[
\mu=\left(\Pi_{n+n'}\beta\right).\mu=\int\left(\Pi_{n+n'}\beta_{\omega}\right).\mu\:d\beta(\omega).
\]
Hence, from (\ref{eq:lb by exact dim}) and by the concavity of conditional
entropy,
\[
\dim\mu+\delta\ge\int\frac{1}{Mn}H\left(\left(\Pi_{n+n'}\beta_{\omega}\right).\mu,\mathcal{D}_{Mn+\chi(n+n')}\mid\mathcal{D}_{\chi(n+n')}\right)\:d\beta(\omega).
\]
Together with Lemma \ref{lem:0-ent lemma}, this gives
\[
\dim\mu+2\delta\ge\int\frac{1}{Mn}H\left(\left(\Pi_{n+n'}\beta_{\omega}\right).\mu,\mathcal{D}_{Mn+\chi(n+n')}\right)\:d\beta(\omega).
\]
Thus, setting
\[
g(\omega):=\frac{1}{Mn}H\left(\left(\Pi_{n+n'}\beta_{\omega}\right).\mu,\mathcal{D}_{Mn+\chi(n+n')}\right)\text{ for }\omega\in\Lambda^{\mathbb{N}},
\]
we have
\begin{equation}
\dim\mu+2\delta\ge\int g(\omega)\:d\beta(\omega).\label{eq:lb on dim(mu) by g}
\end{equation}

For $l\in\mathbb{Z}_{>0}$, let $\mathcal{U}_{l}$ be the set of words
$u\in\Lambda^{l}$ such that $2^{-l(\chi+\delta)}\le\left|\varphi_{u}'(0)\right|\le2^{-l(\chi-\delta)}$.
Let $E$ be the set of $\omega\in\Lambda^{\mathbb{N}}$ such that
$\mathrm{supp}(\beta_{\omega})\subset\Pi^{-1}\left(\Pi\omega\right)$,
$\frac{1}{n}H\left(\beta_{\omega},\mathcal{C}_{n}\right)>\Delta$,
$\beta_{\omega}\left(\left[\mathcal{U}_{n}\right]\right)>1-\delta$,
and $\sigma^{n}\beta_{\omega}\left(\left[\mathcal{U}_{n'}\right]\right)>1-\delta$.
By (\ref{eq:ent of slices}) and (\ref{eq:chi as a.s.  limit}), by
basic properties of disintegrations, and since $\Delta^{-1},\delta^{-1}\ll n$,
we may assume that $\beta(E)>1-\delta$. In what follows, fix $\omega\in E$.

By the definition of $\mathcal{U}_{n}$, there exists a partition
$\left\{ \mathcal{U}_{n,j}\right\} _{j\in J}$ of $\mathcal{U}_{n}$
such that $|J|\le3\delta n$ and,
\begin{equation}
\left|\varphi_{u_{1}}'(0)\right|\le2\left|\varphi_{u_{2}}'(0)\right|\text{ for all }j\in J\text{ and }u_{1},u_{2}\in\mathcal{U}_{n,j}.\label{eq:def prop of U_n,j}
\end{equation}
For $j\in J$ and $u\in\mathcal{U}_{n'}$, set $F_{j,u}:=\left[\mathcal{U}_{n,j}\right]\cap\sigma^{-n}\left[u\right]$
and
\[
g(\omega,j,u):=\frac{1}{Mn}H\left(\left(\Pi_{n+n'}\left(\beta_{\omega}\right)_{F_{j,u}}\right).\mu,\mathcal{D}_{Mn+\chi(n+n')}\right).
\]

Since $\omega\in E$, there exist $0\le\delta'<2\delta$ and $\theta\in\mathcal{M}\left(\Lambda^{\mathbb{N}}\right)$
such that
\[
\beta_{\omega}=\sum_{(j,u)\in J\times\mathcal{U}_{n'}}\beta_{\omega}\left(F_{j,u}\right)\left(\beta_{\omega}\right)_{F_{j,u}}+\delta'\theta.
\]
Thus, by concavity,
\begin{equation}
g(\omega)\ge\sum_{(j,u)\in J\times\mathcal{U}_{n'}}\beta_{\omega}\left(F_{j,u}\right)g(\omega,j,u).\label{eq:lb on g(omega) by g(omega,j,u)}
\end{equation}
Moreover, from $\frac{1}{n}H\left(\beta_{\omega},\mathcal{C}_{n}\right)>\Delta$,
by the convexity bound (see \cite[Lemma 3.1]{Ho1}), and since the
cardinality of $J\times\mathcal{U}_{n'}$ is at most $3\delta n|\Lambda|^{n'}$,
\begin{multline*}
\Delta<\sum_{(j,u)\in J\times\mathcal{U}_{n'}}\beta_{\omega}\left(F_{j,u}\right)\frac{1}{n}H\left(\left(\beta_{\omega}\right)_{F_{j,u}},\mathcal{C}_{n}\right)\\
+\delta'\frac{1}{n}H\left(\theta,\mathcal{C}_{n}\right)+\frac{\log\left(3\delta n\right)}{n}+\frac{n'}{n}\log|\Lambda|.
\end{multline*}
Note that,
\begin{equation}
\frac{1}{n}H\left(\theta',\mathcal{C}_{n}\right)\le\log|\Lambda|\text{ for all }\theta'\in\mathcal{M}\left(\Lambda^{\mathbb{N}}\right).\label{eq:ub ent all measures}
\end{equation}
Hence, from the previous formula and the definition of $n'$,
\begin{equation}
\Delta/3<\sum_{(j,u)\in J\times\mathcal{U}_{n'}}\beta_{\omega}\left(F_{j,u}\right)\frac{1}{n}H\left(\left(\beta_{\omega}\right)_{F_{j,u}},\mathcal{C}_{n}\right).\label{eq:Delta/3<}
\end{equation}

Let $\mathcal{Q}$ be the set of all $(j,u)\in J\times\mathcal{U}_{n'}$
such that $\frac{1}{n}H\left(\left(\beta_{\omega}\right)_{F_{j,u}},\mathcal{C}_{n}\right)\ge\Delta/6$.
From (\ref{eq:ub ent all measures}) and (\ref{eq:Delta/3<}),
\begin{equation}
\frac{\Delta}{6\log|\Lambda|}<\sum_{(j,u)\in\mathcal{Q}}\beta_{\omega}\left(F_{j,u}\right).\label{eq:lb on mass of (j,u ) in Q}
\end{equation}

Recall that $C>1$ is a large global constant depending only on $\Phi$.
For $v\in\mathcal{U}_{n}$ and $w\in\mathcal{U}_{n'}$, it follows
by bounded distortion, the chain rule, and the definition of the sets
$\mathcal{U}_{l}$, that 
\[
C^{-1}2^{-\delta(n+n')}\le2^{\chi(n+n')}\left|\varphi_{vw}'(x)\right|\le C2^{\delta(n+n')}\text{ for }x\in I.
\]
In particular, $S_{2^{\chi(n+n')}}\circ\varphi_{vw}$ is bi-Lipschitz
with bi-Lipschitz constant at most $C2^{\delta(n+n')}$. Moreover,
given $(j,u)\in J\times\mathcal{U}_{n'}$ and setting $\xi:=\left(\Pi_{n+n'}\left(\beta_{\omega}\right)_{F_{j,u}}\right)$,
we have
\[
\mathrm{supp}(\xi)\subset\left\{ \varphi_{vw}\::\:v\in\mathcal{U}_{n}\text{ and }w\in\mathcal{U}_{n'}\right\} .
\]
From these facts, by concavity, from Lemma \ref{lem:dyad ent =000026 lip func},
by (\ref{eq:ent dim =00003D exact dim}), and since $C,\delta^{-1}\ll n$,
\[
g(\omega,j,u)\ge\int\frac{1}{Mn}H\left(S_{2^{\chi(n+n')}}\varphi\mu,\mathcal{D}_{Mn}\right)\:d\xi(\varphi)-O(1/n)\ge\dim\mu-O(\delta).
\]
We have thus shown that,
\begin{equation}
g(\omega,j,u)\ge\dim\mu-O(\delta)\text{ for }(j,u)\in J\times\mathcal{U}_{n'}.\label{eq:simple lb by conc}
\end{equation}

Fix $(j,u)\in\mathcal{Q}$, and set $a:=\varphi_{u}(0)$. Recall from
Section \ref{subsec:Function-spaces} that the linear operator $P_{k,a}:\mathcal{A}(I)\rightarrow\mathcal{P}_{k}$
is defined by sending $f\in\mathcal{A}(I)$ to its $k$-th order Taylor
polynomial at the point $a$. Let $v\in\mathcal{U}_{n,j}$ and $x\in I$
be given. By Taylor's theorem with Lagrange remainder term,
\[
\left|\left(\varphi_{v}-P_{k,a}\varphi_{v}\right)\left(\varphi_{u}(x)\right)\right|\le\frac{1}{(k+1)!}\left|\varphi_{u}(x)-a\right|^{k+1}\sup_{y\in I}\left|\varphi_{v}^{(k+1)}(y)\right|.
\]
Hence, from Lemma \ref{lem:bd distort}, by the mean value theorem,
and since $v\in\mathcal{U}_{n}$ and $u\in\mathcal{U}_{n'}$,
\begin{eqnarray*}
\left|\left(\varphi_{v}-P_{k,a}\varphi_{v}\right)\left(\varphi_{u}(x)\right)\right| & \le & C^{k+1}\left|\varphi_{u}'(0)\right|^{k+1}C^{k+1}\left|\varphi_{v}'(0)\right|\\
 & \le & C^{2k+2}2^{(k+1)n'(\delta-\chi)}2^{n(\delta-\chi)}.
\end{eqnarray*}
From this, since $\Delta^{-1},C,M\ll k\ll n$, by the definition of
$n'$, and since we can assume that $\delta<\chi/2$,
\begin{equation}
\left|\left(\varphi_{v}-P_{k,a}\varphi_{v}\right)\left(\varphi_{u}(x)\right)\right|\le2^{-Mn-\chi(n+n')}\text{ for }v\in\mathcal{U}_{n,j}\text{ and }x\in I.\label{eq:bd on remainder}
\end{equation}

Note that
\[
\left(\Pi_{n+n'}\left(\beta_{\omega}\right)_{F_{j,u}}\right).\mu=\left(\Pi_{n}\left(\beta_{\omega}\right)_{F_{j,u}}\right).\varphi_{u}\mu,
\]
and also
\begin{equation}
\mathrm{supp}\left(\Pi_{n}\left(\beta_{\omega}\right)_{F_{j,u}}\right)=\left\{ \varphi_{v}\::\:v\in\mathcal{U}_{n,j}\text{ and }\beta_{\omega}\left(\left[vu\right]\right)>0\right\} .\label{eq:form of supp of push by Pi_n}
\end{equation}
Together with (\ref{eq:bd on remainder}) and Lemma \ref{lem:cont of ent wrt close maps},
this implies that
\begin{equation}
g(\omega,j,u)\ge\frac{1}{Mn}H\left(\left(P_{k,a}\Pi_{n}\left(\beta_{\omega}\right)_{F_{j,u}}\right).\varphi_{u}\mu,\mathcal{D}_{Mn+\chi(n+n')}\right)-\delta.\label{eq:lb on g by taylor}
\end{equation}

Let $\psi:\mathbb{R}\rightarrow\mathbb{R}$ be the affine map such
that $\psi(a)=0$ (recall that $a:=\varphi_{u}(0)$) and $\psi\left(\varphi_{u}(1)\right)=1$.
In particular, $\psi\circ\varphi_{u}(I)=I$. Let $r_{\psi},t_{\psi}\in\mathbb{R}$
be with $\psi(x)=r_{\psi}x+t_{\psi}$ for $x\in\mathbb{R}$. By the
mean value theorem, by bounded distortion, and since $u\in\mathcal{U}_{n'}$,
\begin{equation}
C^{-1}2^{n'(\chi-\delta)}\le C^{-1}\left|\varphi_{u}'(0)\right|^{-1}\le\left|r_{\psi}\right|\le C\left|\varphi_{u}'(0)\right|^{-1}\le C2^{n'(\chi+\delta)}.\label{eq:bounds on r_psi}
\end{equation}
Let $R_{\psi^{-1}}:\mathcal{P}_{k}\rightarrow\mathcal{P}_{k}$ be
with $R_{\psi^{-1}}(p):=p\circ\psi^{-1}$ for $p\in\mathcal{P}_{k}$.

By (\ref{eq:def prop of U_n,j}) and $\mathcal{U}_{n,j}\subset\mathcal{U}_{n}$,
there exists $2^{n(\chi-\delta)}\le\alpha\le2^{n(\chi+\delta)}$ such
that
\begin{equation}
\frac{1}{2}\le\alpha\left|\varphi_{v}'(0)\right|\le2\text{ for }v\in\mathcal{U}_{n,j}.\label{eq:bounds on alpha}
\end{equation}
Set
\[
\nu:=S_{\alpha|r_{\psi}|}T_{-\Pi(\omega)}R_{\psi^{-1}}P_{k,a}\Pi_{n}\left(\beta_{\omega}\right)_{F_{j,u}},
\]
where recall from Section \ref{subsec:Basic-notations} that for $p(X)\in\mathcal{P}_{k}$
\[
T_{-\Pi(\omega)}\left(p(X)\right)=p(X)-\Pi(\omega)\:\text{ and }\:S_{\alpha|r_{\psi}|}\left(p(X)\right)=\alpha|r_{\psi}|\cdot p(X).
\]
Note that,
\[
\nu.\left(\psi\varphi_{u}\mu\right)=S_{\alpha|r_{\psi}|}T_{-\Pi(\omega)}\left(\left(P_{k,a}\Pi_{n}\left(\beta_{\omega}\right)_{F_{j,u}}\right).\varphi_{u}\mu\right).
\]
Together with (\ref{eq:lb on g by taylor}), $|r_{\psi}|\le C2^{n'(\chi+\delta)}$
and $\alpha\le2^{n(\chi+\delta)}$, this gives
\begin{equation}
g(\omega,j,u)\ge\frac{1}{Mn}H\left(\nu.\left(\psi\varphi_{u}\mu\right),\mathcal{D}_{Mn}\right)-O(\delta).\label{eq:lb on g(omega,j,u) bef ent enc}
\end{equation}

Next, we aim to apply Theorem \ref{thm:ent inc with poly} to the
entropy appearing on the right-hand side of the last inequality. We
proceed to verify the conditions of the theorem. Let $v\in\mathcal{U}_{n,j}$
be with $\beta_{\omega}\left(\left[vu\right]\right)>0$, and set $p:=S_{\alpha|r_{\psi}|}T_{-\Pi(\omega)}R_{\psi^{-1}}P_{k,a}\varphi_{v}$.
Note that by (\ref{eq:form of supp of push by Pi_n}), each polynomial
in $\mathrm{supp}(\nu)$ is of this form. We have
\[
\psi^{-1}(X)-a=r_{\psi}^{-1}X-t_{\psi}r_{\psi}^{-1}-a=r_{\psi}^{-1}\left(X-\psi(a)\right)=r_{\psi}^{-1}X,
\]
which, together with (\ref{eq:exp form of P_k,a}), gives
\begin{eqnarray}
p(X)+\alpha|r_{\psi}|\Pi(\omega) & = & \alpha|r_{\psi}|\sum_{l=0}^{k}\frac{\varphi_{v}^{(l)}\left(a\right)}{l!}\left(\psi^{-1}(X)-a\right)^{l}\nonumber \\
 & = & \alpha|r_{\psi}|\sum_{l=0}^{k}\frac{\varphi_{v}^{(l)}\left(a\right)}{l!}r_{\psi}^{-l}X^{l}.\label{eq:dev of p}
\end{eqnarray}
Thus,
\begin{equation}
\Vert p(X)\Vert_{2}^{2}=\alpha^{2}r_{\psi}^{2}\left(\varphi_{v}\left(a\right)-\Pi(\omega)\right)^{2}+\sum_{l=1}^{k}\left(\alpha r_{\psi}\frac{\varphi_{v}^{(l)}\left(a\right)}{l!}r_{\psi}^{-l}\right)^{2}.\label{eq:norm of p equals}
\end{equation}

By Lemma \ref{lem:bd distort}, since $v\in\mathcal{U}_{n,j}$, and
from (\ref{eq:bounds on alpha}),
\begin{equation}
\left|\varphi_{v}^{(l)}\left(a\right)\right|\le l!C^{l}\left|\varphi_{v}'(0)\right|\le l!C^{l}2\alpha^{-1}\text{ for }1\le l\le k.\label{eq:ub on derivatives}
\end{equation}
Since $\omega\in E$, we have $\mathrm{supp}(\beta_{\omega})\subset\Pi^{-1}\left(\Pi\omega\right)$.
From this and $\beta_{\omega}\left(\left[vu\right]\right)>0$, it
follows that there exists $\omega'\in\left[vu\right]$ such that
\[
\Pi(\omega)=\Pi(\omega')=\varphi_{vu}\Pi(\sigma^{n+n'}(\omega')).
\]
Additionally, from $a=\varphi_{u}(0)$, we get $\varphi_{v}\left(a\right)=\varphi_{vu}\left(0\right)$.
Hence, by the mean value theorem, by bounded distortion, since $\Pi(\sigma^{n+n'}(\omega'))\in I$,
and from (\ref{eq:bounds on r_psi}) and (\ref{eq:bounds on alpha}),
\begin{multline*}
\left|\varphi_{v}\left(a\right)-\Pi(\omega)\right|=\left|\varphi_{vu}\left(0\right)-\varphi_{vu}\Pi(\sigma^{n+n'}(\omega'))\right|\\
\le C\left|\varphi_{v}'(0)\right|\left|\varphi_{u}'(0)\right|\le2C^{2}\alpha^{-1}|r_{\psi}|^{-1}.
\end{multline*}
From this, (\ref{eq:norm of p equals}) and (\ref{eq:ub on derivatives}),
\[
\Vert p(X)\Vert_{2}^{2}\le4C^{4}+\sum_{l=1}^{k}4C^{2l}|r_{\psi}|^{2-2l}\le8kC^{2k}.
\]

Let $x\in I$ be given. From (\ref{eq:dev of p}),
\[
p'(x)=\alpha|r_{\psi}|\sum_{l=1}^{k}\frac{\varphi_{v}^{(l)}\left(a\right)}{(l-1)!}r_{\psi}^{-l}x^{l-1}.
\]
From (\ref{eq:bounds on r_psi}) and (\ref{eq:ub on derivatives}),
it follows that for $2\le l\le k$
\[
\left|\alpha\frac{\varphi_{v}^{(l)}\left(a\right)}{(l-1)!}r_{\psi}^{1-l}x^{l-1}\right|\le2kC^{2k-1}\cdot2^{-n'(\chi-\delta)}.
\]
Moreover, by (\ref{eq:bounds on alpha}) and bounded distortion, we
get $\left|\alpha\varphi_{v}'\left(a\right)\right|\ge\frac{1}{2}C^{-1}$.
Since $C,k\ll n'$ and we can assume that $\delta<\chi/2$, all of
this implies that $\left|p'(x)\right|\ge\frac{1}{4}C^{-1}$. We have
thus shown that,
\begin{equation}
\Vert p\Vert_{2}\le\left(8kC^{2k}\right)^{1/2}\text{ and }\left|p'(x)\right|\ge\frac{1}{4}C^{-1}\text{ for }p\in\mathrm{supp}(\nu)\text{ and }x\in I.\label{eq:ub on norm =000026 lb on der}
\end{equation}

Let $v_{1},v_{2}\in\mathcal{U}_{n,j}$ be distinct, and for $i=1,2$
set $p_{i}:=S_{\alpha|r_{\psi}|}T_{-\Pi(\omega)}R_{\psi^{-1}}P_{k,a}\varphi_{v_{i}}$.
Assume by contradiction that $\mathcal{D}_{Mn}^{\mathcal{P}_{k}}(p_{1})=\mathcal{D}_{Mn}^{\mathcal{P}_{k}}(p_{2})$.
By the definition of $\mathcal{D}_{Mn}^{\mathcal{P}_{k}}$ (see Section
\ref{subsec:Dyadic-partitions}), this implies that
\begin{equation}
\left|p_{1}(x)-p_{2}(x)\right|\le(k+1)2^{-Mn}\text{ for }x\in I.\label{eq:ub on dist of poly}
\end{equation}

On the other hand, by (\ref{eq:by exp sep}), we have $\Vert\varphi_{v_{1}u}-\varphi_{v_{2}u}\Vert_{I}\ge c^{n+n'}$.
Thus, since $\psi^{-1}(I)=\varphi_{u}(I)$, there exists $x_{0}\in I$
such that
\[
\left|\varphi_{v_{1}}\left(\psi^{-1}(x_{0})\right)-\varphi_{v_{2}}\left(\psi^{-1}(x_{0})\right)\right|\ge c^{n+n'}.
\]
Additionally, from $\psi^{-1}(I)=\varphi_{u}(I)$ and (\ref{eq:bd on remainder}),
\[
\left|\left(\varphi_{v_{i}}-P_{k,a}\varphi_{v_{i}}\right)\left(\psi^{-1}(x_{0})\right)\right|\le2^{-Mn-\chi(n+n')}\text{ for }i=1,2.
\]
From the last two inequalities and since $c^{-1}\ll M$,
\[
\left|P_{k,a}\varphi_{v_{1}}\left(\psi^{-1}(x_{0})\right)-P_{k,a}\varphi_{v_{2}}\left(\psi^{-1}(x_{0})\right)\right|\ge c^{n+n'}/2,
\]
which gives $\left|p_{1}(x_{0})-p_{2}(x_{0})\right|\ge c^{n+n'}/2$.
Since $c^{-1}\ll M$ and $k\ll n$, this contradicts (\ref{eq:ub on dist of poly}).
Hence we must have $\mathcal{D}_{Mn}^{\mathcal{P}_{k}}(p_{1})\ne\mathcal{D}_{Mn}^{\mathcal{P}_{k}}(p_{2})$,
which implies
\begin{equation}
\frac{1}{Mn}H\left(\nu,\mathcal{D}_{Mn}^{\mathcal{P}_{k}}\right)=\frac{1}{Mn}H\left(\left(\beta_{\omega}\right)_{F_{j,u}},\mathcal{C}_{n}\right)\ge\frac{\Delta}{6M},\label{eq:lb on ent of nu}
\end{equation}
where the last inequality follows from $(j,u)\in\mathcal{Q}$.

For $x\in I$ we have $(\psi\circ\varphi_{u})'(x)=r_{\psi}\varphi_{u}'(x)$.
Thus, by (\ref{eq:bounds on r_psi}) and bounded distortion,
\[
C^{-2}\le\left|(\psi\circ\varphi_{u})'(x)\right|\le C^{2}.
\]
From this, from (\ref{eq:ub on norm =000026 lb on der}) and (\ref{eq:lb on ent of nu}),
since $\dim\mu<1$, by the relations (\ref{eq:rel between params in main pf}),
and by Theorem \ref{thm:ent inc with poly},
\[
\frac{1}{Mn}H\left(\nu.\left(\psi\varphi_{u}\mu\right),\mathcal{D}_{Mn}\right)\ge\dim\mu+\rho.
\]
Hence, by (\ref{eq:lb on g(omega,j,u) bef ent enc}),
\begin{equation}
g(\omega,j,u)\ge\dim\mu+\rho-O(\delta)\text{ for all }(j,u)\in\mathcal{Q}.\label{eq:lb for (j,u) in Q}
\end{equation}

We can now complete the proof. From (\ref{eq:lb on g(omega) by g(omega,j,u)}),
(\ref{eq:simple lb by conc}), and (\ref{eq:lb for (j,u) in Q}),
\begin{eqnarray*}
g(\omega) & \ge & \sum_{(j,u)\in\mathcal{Q}}\beta_{\omega}\left(F_{j,u}\right)\left(\dim\mu+\rho-O(\delta)\right)\\
 & + & \sum_{(j,u)\in(J\times\mathcal{U}_{n'})\setminus\mathcal{Q}}\beta_{\omega}\left(F_{j,u}\right)\left(\dim\mu-O(\delta)\right).
\end{eqnarray*}
Thus, by (\ref{eq:lb on mass of (j,u ) in Q}) and since $\cup_{(j,u)\in J\times\mathcal{U}_{n'}}F_{j,u}=\left[\mathcal{U}_{n}\right]\cap\sigma^{-n}\left[\mathcal{U}_{n'}\right]$,
\[
g(\omega)\ge\frac{\rho\Delta}{6\log|\Lambda|}+\beta_{\omega}\left(\left[\mathcal{U}_{n}\right]\cap\sigma^{-n}\left[\mathcal{U}_{n'}\right]\right)\left(\dim\mu-O(\delta)\right),
\]
which holds for all $\omega\in E$. From this and by the definition
of $E$,
\[
g(\omega)\ge\frac{\rho\Delta}{6\log|\Lambda|}+\dim\mu-O(\delta)\text{ for all }\omega\in E.
\]
Hence, from (\ref{eq:lb on dim(mu) by g}) and since $\beta(E)>1-\delta$,
\[
\dim\mu+2\delta\ge(1-\delta)\left(\frac{\rho\Delta}{6\log|\Lambda|}+\dim\mu-O(\delta)\right).
\]
But since $\Delta^{-1},\rho^{-1}\ll\delta^{-1}$, this yields the
desired contradiction, completing the proof of the theorem.
\end{proof}
\bibliographystyle{plain}
\bibliography{../../bibfile}


$\newline$\textsc{Department of Mathematics, Technion, Haifa, Israel}$\newline$$\newline$\textit{E-mail: }
\texttt{arapaport@technion.ac.il}
\end{document}